\documentclass[letterpaper,11pt]{article}

\usepackage[square,numbers]{natbib}

\usepackage[utf8]{inputenc}   
\usepackage[T1]{fontenc}  

\usepackage{pgf,tikz}
\usetikzlibrary{arrows}
\usepackage{multicol}
\usepackage{amsthm}
\usepackage{amsmath}
\usepackage{amssymb}
\usepackage{amsfonts}
\usepackage{stmaryrd}
\usepackage{mathabx}
\usepackage{latexsym}
\usepackage{color}
\usepackage{graphics,graphicx,graphpap}
\usepackage{multirow}
\usepackage{rotating}
\usepackage[new]{old-arrows}
\usepackage[all]{xy}
\usepackage[letterpaper]{geometry}
\usepackage{subfigure}
\usepackage{hyperref}
\usepackage{orcidlink}
\usepackage{authblk}
\theoremstyle{plain}

\usepackage{setspace}

\usepackage{tocbibind}

\DeclareMathAlphabet{\mathpzc}{OT1}{pzc}{m}{it}
\newtheorem{theorem}{Theorem}
\newtheorem{prop}[theorem]{Proposition}
\newtheorem{cor}[theorem]{Corollary}
\newtheorem{lem}[theorem]{Lemma}
\newtheorem{que}[theorem]{Question}
\newtheorem{exa}[theorem]{Example}

\newcommand{\hocolim}{\operatornamewithlimits{\mathrm{hocolim}}}
\newcommand{\colim}{\operatornamewithlimits{\mathrm{colim}}}
\newcommand*{\email}[1]{%
    \normalsize\href{mailto:#1}{#1}\par
    }

\title{The Forest Filtration of a Graph}
\author{Andr\'es Carnero Bravo\;\orcidlink{0009-0001-0221-5908}}
\affil{Centro de Ciencias Matemáticas, UNAM\\ \email{carnero@matmor.unam.mx}}
\begin{document}
\maketitle
\begin{abstract}
    Given a graph $G$, we define a filtration of simplicial complexes associated to $G$, 
    $\mathcal{F}_0(G)\subseteq\mathcal{F}_1(G)\subseteq\cdots\subseteq\mathcal{F}_\infty(G)$ 
    where the first complex is the independence complex and the last the complex is formed by the 
    acyclic sets of vertices. We prove some properties of this filtration and we calculate 
    the homotopy type for various families of graphs. We give an upper
    bound for the decycling number and generalizations of this parameter using the dimensions of the rational 
    cohomology groups of these complexes. We also derive an upper bound for the Fibonacci numbers of ternary graphs.
\end{abstract}
\tableofcontents
\textbf{Acknowledgments.} This work was supported by UNAM Posdoctoral Program (POSDOC).
The author wishes to thank Omar Antol\'in-Camarena for his comments and suggestions which improved this paper.
\section{Introduction}
Given a simple graph $G$, the sets of edges that induce acyclic graphs form a simplicial complex, 
the independence complex of the graphic matroid associated to the graph. This complex is pure of dimension $n-k(G)-1$ and 
has the homotopy type of the wedge of $T_G(0,1)$ spheres, where $T_G(x,y)$ is the Tutte polynomial of the 
graph (see \citep{bjornermatroidshella}) and $k(G)$ is the number of connected components of the graph. 
If instead of edges we take vertices, we get a complex $\mathcal{F}_\infty(G)$ which does 
not have to be pure and can, for example, have the homotopy type of a wedge of spheres of  different dimension, as it does for the complex associated to the $k$-complete 
multipartite graph for $k\geq3$. This complex, defined in \citep{tesiszuffi} but not studied, 
is part of a filtration of complexes associated to the graph:
\begin{equation}\tag{$i$}\label{filt}
\mathcal{F}_0(G)\subseteq\mathcal{F}_1(G)\subseteq\mathcal{F}_2(G)\subseteq\cdots\subseteq\mathcal{F}_\infty(G),
\end{equation}
where for each $d$ we take as simplices the set of vertices which induce acyclic graphs of maximum degree at most $d$. It is worth to say 
something about the first two complexes of the filtration which have been studied before. The case $d=0$ is 
the independence complex which has been studied extensively 
\citep{adamsplit,barmak,Ehrenborg_2006,engstrom09,haxell} and, for example, in \citep{Gonz_lez_Meneses_2017,Przytycki_2017} 
it was used to study the extreme Khovanov (co)homology of link diagrams. Because of the difficulty of calculating homotopy types of 
independence complexes most of the work has been focused on specific families of graphs  
\citep{indcomplcartprod,Bousquet_M_lou_2007,Braun_2017,homotopygoyal,Iriye_2012,Jonsson_2009}.
The case $d=1$ is also a member of a different family generalizing the independence complex, called the $k$-independence complexes. Our 
$\mathcal{F}_1(G)$ is isomorphic to the $2$-independence complex, which features in \citep{salvetti2015,salvetti2018}, where it was used to 
study the local homology of Artin groups. 
In \citep{taylan}, a simplicial complex $\mathcal{D}(G,H)$ is defined for any graph $G$ and any family of graphs 
$\mathcal{H}$, whose simplices are the subsets of vertices of $G$ that do not induce a graph in $\mathcal{H}$. From this point of view, if 
$F_i$ is the family of forests of maximum degree $i$, then $\mathcal{D}(G,F_{d+1})$  is precisely $\mathcal{F}_i(G)$, but these families 
were not studied in that article.
To the best of our knowledge, the cases $2\leq d<\infty$ have not been studied before.

We will determine some basic properties of the filtration (\ref{filt}) and study the homotopy type for some 
families (paths, cycles, double stars and cactus graphs) and some graph products, in particular we will study the homotopy type 
for graph joins which will allow us to calculate the homotopy type for some complete multipartite graphs and wheels. Using rational 
cohomology, we will give upper bounds for some combinatorial parameters of the graph;
two of these parameters have already been extensively studied: 
the vertex cover number \citep{bazzi,dinur,jinminvercov,matsudomvercov23,matsudomvercov25} and the decycling number 
\citep{bau,beineke,engstromupperbounds,erdosindcirc,ren}. The other parameters are generalizations of these two in the context of the 
corresponding complex of the filtration. The bound for the vertex cover number will allow us to get 
an upper bound for the Fibonacci numbers of ternary graphs (graphs without induced cycles of length divisible by $3$); this bound will be 
given in terms of combinatorial parameters.
The Fibonacci number of a graph is the number of its independent sets, and it has been extensively studied  
\citep{kirschfiboii,kirschfiboiii,linfibo,perdersenunpp,prodingerfibo,startekfibo}-- it is worth mentioning that this number is also called 
the Merrifield-Simmons index in mathematical chemistry \citep{deitamethcommerr,linotemerrsimind,wagnermaximinhoymerr,wagnerexttree}.

In Section $2$ we recall basic definitions and results from graph theory and algebraic topology. 
There we define and recall some properties of homotopy colimits, which are one of our main tools. In Section $3$ we define 
the filtration (\ref{filt}) and prove some of its main properties. Section $4$  is devoted to the computation of the homotopy type of the 
complexes of the filtration. In the last section, we provide upper bounds for the decycling number, the vertex cover number, and some 
generalizations of these two.
\section{Preliminaries}
We only consider simple graphs, so no loops or multiedges are allowed. Given a graph $G$, $V(G)$ will be the set of vertices 
and $E(G)$ the set of edges. For $S\subseteq V(G)$, $G[S]$ is the \textit{graph induced by $S$}, this is the graph 
with vertex set $S$, and two vertices are adjacent if and only if they are adjacent in $G$. For a vertex $v$, 
$N_G(v)=\{u\in V(G):\;uv\in E(G)\}$ is its \textit{open neighborhood}, we omit the subindex $G$ if there is no risk of confusion. 
The\textit{degree} of a vertex will be denoted $d_G(v)=|N_G(v)|$. The \textit{maximum} and \textit{minimum degrees} will be denoted 
$\Delta(G)$ and $\delta(G)$ respectively.  
Given a graph $G$, its \textit{girth} $g(G)$ is the length of the smallest cycle in $G$, if there is no cycle in $G$ we take $g(G)=\infty$. 
If $G$ has no cycles, we say it is a \textit{forest}. A \textit{tree} is a connected forest.

Given a graph $G$ and $d$ a non-negative integer, we define 
$$t_d(G)=\max_{S\subseteq V(G)}\{|S|:\;G[S] \mbox{ is a forest of maximum degree at most }d\}$$
and 
$$\nabla_d(G)=\min_{S\subseteq V(G)}\{|S|:\;G-S\mbox{ is a forest of maximum degree at most d}\}$$
By removing the bound on the maximum degree, we define numbers $t(G)$ and $\nabla(G)$ in a similar way. From the definition 
it is easy to see that if $G$ has order $n$, then $\nabla_d(G)=n-t_d(G)$. If $d=0$, then $t_0(G)$ and $\nabla_0(G)$ are the 
\textit{independence number} and the \textit{vertex covering number}, respectively. We will denote the independence number by $\alpha(G)$ 
and the vertex covering number by $\tau(G)$.
The number $\nabla(G)$ is known as the \textit{decycling number} or as the \textit{vertex feed-back number}, with a vertex set $S$ such 
that $G-S$ is a forest called a \textit{decycling set}.

Given two graphs $G$ and $H$, we consider the following two graphs whose vertex sets are $V(G)\times V(H)$:
\begin{enumerate}
    \item The \textit{cartesian product} $G\oblong H$, where $\{(u_1,v_1)(u_2,v_2)\}$ is an edge if $u_1=u_2$ and $v_1v_2\in E(H)$ or 
    $u_1u_2\in E(G)$ and $v_1=v_2$.
    \item The \textit{categorical product} $G\times H$, where $\{(u_1,v_1)(u_2,v_2)\}$ is an edge if $u_1u_2\in E(G)$ and $v_1v_2\in E(H)$.
\end{enumerate}

For all the graph definitions not stated here we follow \citep{graphsanddigraphs}.

A \textit{simplicial complex} $K$ on the vertex set $V$ is a family of subsets, called \textit{simplices}, 
such that if $\sigma$ is in $K$, any subset $\tau\subseteq\sigma$ is also in $K$. Notice that we consider the empty set as a 
simplex and we are allowing phantom vertices, \textit{i.e.} we allow to exist vertices $v$ in $V$ such that $\{v\}$ 
is not a simplex of $K$. 
We say that a simplex $\sigma$ has \textit{dimension} $|\sigma|-1$ and the maximum of the dimension is the dimension of the complex which 
we will denote as $\mathrm{dim}(K)$. 
The \textit{$q$-skeleton} of a complex $K$, denoted $\mathrm{sk}_qK$, is the subcomplex of all simplices with at most $q+1$ elements. 
We will write $f_r(K)$ for the number of simplices of dimension $r$ in $K$. Notice that $f_{-1}(K)=1$ if $K\neq\emptyset$.

Given a simplex $\sigma\in K$, its \textit{link} is the simplicial complex 
$\mathrm{lk}(\sigma)=\{\tau\subseteq V\colon\tau\cap\sigma=\emptyset \;\;\mbox{and}\;\;\tau\cup\sigma\in K\}$ and its \textit{star} is
$\mathrm{st}(\sigma)=\{\tau\in K:\tau\cup\sigma\in K\}$. For a vertex we will write $\mathrm{lk}(v)$ and $\mathrm{st}(v)$ 
instead of $\mathrm{lk}(\{v\})$ or $\mathrm{st}(\{v\})$.

Given two simplicial complexes $K$ and $L$ with disjoint vertex sets, their \textit{join} 
is the simplicial complex $K*L=\{\sigma\cup\tau:\;\sigma\in K\;\mbox{ and }\;\tau\in L\}$, this definition depends on the fact that the 
empty set is a simplex of non-empty simplicial complexes. The \textit{cone} of  
$K$ is the simplicial complex $C(K)$ obtained as the join $\{\{v_0\}\}*K$ with $v_0$ not a vertex of $K$, we call this vertex \textit{apex 
vertex}.

Given a set family $\mathcal{U}$, its \textit{nerve} is the simplicial complex $\mathcal{N}(\mathcal{U})$ with vertex set 
$\mathcal{U}$, and $U_1,\dots,U_l$ form a simplex if $U_1\cap\cdots\cap U_l\neq\emptyset$.  

\begin{theorem}[Nerve Theorem, see {\citep[Theorem 10.6]{bjornertopmeth}}]
Let $K$ be a simplicial complex and $\{U_1,\dots,U_r\}$ a family of subcomplexes such that $K=U_1\cup\cdots\cup U_r$. If every nonempty 
finite intersection $U_{i_1}\cap\cdots\cap U_{i_l}$ is contractible. Then $K\simeq\mathcal{N}(\mathcal{U})$
\end{theorem}

We will not distinguish between a complex and its geometric realization. All the homology and cohomology groups will be with 
integer coefficients until the last section, where the coefficients will be the rational numbers.

\begin{theorem}[Whitehead's theorem, see {\citep[Corollary 4.33]{hatcher}}]\label{whiteheadhomologia}
If $X$ and $Y$ are simply connected CW-complexes and there is a continuous map $f\colon X\longrightarrow Y$ such 
that $f_*:H_n(X)\longrightarrow H_n(Y)$ is an isomorphism for each $n$, then $f$ is an homotopy equivalence.
\end{theorem}

Given a complex $X$ on $n$ vertices, its \textit{Alexander Dual} is the complex 
$$X^*=\{\sigma\subseteq V(X)\colon\;V(X)-\sigma\notin X\}.$$

\begin{theorem}\label{dualidadalexander}(see \citep{bjorneralexander})
Let $X$ be a simplicial complex with $n$ vertices, then
$$\tilde{H}_i(X)\cong\tilde{H}^{n-i-3}(X^*)$$
\end{theorem}

We will use the following folklore result (see \citep{MR4715353} for a proof).
\begin{theorem}\label{cwsimplcon}
If $X$ is a simply connected CW-complex such that $\tilde{H}_q(X)\cong\mathbb{Z}^a$ for some $q\geq2$ and the rest of the homology groups 
are trivial, then $X\simeq\displaystyle\bigvee_a\mathbb{S}^q$.
\end{theorem}

In the same vein of last theorem, we will make use of the following result.

\begin{theorem}\citep[Example 4C.2]{hatcher}\label{gradconse}
If $X$ is an CW-complex simply connected such that $\tilde{H}_q(X)\cong\mathbb{Z}^a$, $\tilde{H}_{q+1}(X)\cong\mathbb{Z}^b$ for some 
$q\geq2$ and the rest of the homology groups are trivial, then $X\simeq\displaystyle\bigvee_a\mathbb{S}^q\vee\bigvee_b\mathbb{S}^{q+1}$.
\end{theorem}

Now we give the basic results of homotopy colimits which will be our main tool. 
For a positive integer $n$, we take $\underline{n}=\{1,\dots,n\}$ and denote its power set as $\mathcal{P}(\underline{n})$. We define
$\mathcal{P}_1(\underline{n})=\mathcal{P}(\underline{n})-\{\underline{n}\}$.
A \textit{punctured $n$-cube} $\mathcal{X}$ consists of:
\begin{itemize}
    \item A topological space $\mathcal{X}(S)$ for each $S$ in $\mathcal{P}_1(\underline{n})$.
    \item A continuous function $f_{S\subseteq T}\colon\mathcal{X}(S)\longrightarrow\mathcal{X}(T)$ for each $S\subseteq T$.
\end{itemize}
These functions are such that $f_{S\subseteq S}=1_{\mathcal{X}(S)}$ and for any $R\subseteq S\subseteq T$ the following diagram commutes
\begin{equation*}
    \xymatrix{
    \mathcal{X}(R) \ar@{->}[r]^{f_{R\subseteq S}} \ar@{->}[dr]_{f_{R\subseteq T}}& \mathcal{X}(S) \ar@{->}[d]^{f_{S\subseteq T}}\\
     & \mathcal{X}(T).
    }
\end{equation*}  
A punctured $n$-cube of interest for a given topological space $X$ is the constant punctured cube $\mathcal{C}_X$, 
where $\mathcal{C}_X(S)=X$ for any set $S$ and all the functions are $1_X$.
The \textit{colimit} of a punctured $n$-cube is the space
$$\colim(\mathcal{X})=\bigsqcup_{S\in\mathcal{P}_1(\underline{n})}\mathcal{X}(S)/\sim$$
where $\sim$ is the equivalence relation generated by $f_{S\subseteq T_1}(x_S)\sim f_{S\subseteq T_2}(x_S)$ for $T_1,T_2$ and $S\subseteq T_1,T_2$. From the definition
is clear that $\colim(\mathcal{C}_X)\cong X$ for any $X$.

For any $n\geq1$ and $S$ in $\mathcal{P}_1(\underline{n})$ we take:
$$\Delta(S)=\left\lbrace(t_1,t_2,\dots,t_n)\in \mathbb{R}^n\colon\;\sum_{i=1}^nt_i=1\mbox{ and }t_i=0\mbox{ for all }i\in S\right\rbrace$$
and $d_{S\subseteq T}\colon \Delta(T)\longrightarrow\Delta(S)$ the corresponding inclusion. Now, for a punctured 
$n$-cube $\mathcal{X}$, its \textit{homotopy colimit} is 
$$\hocolim(\mathcal{X})=\bigsqcup_{S\in\mathcal{P}_1(\underline{n})}\mathcal{X}(S)\times\Delta(S)/\sim$$
where $\sim$ is the equivalence relation generated by $(x_S,d_{S\subseteq T}(t))\sim(f_{S\subseteq T}(x_S),t)$. When $n=2$, 
we will specify the punctured $2$-cube as the diagram
\begin{equation*}
    \xymatrix{
    \mathcal{D}\colon & X\ar@{<-}[r]^{f} & Z \ar@{->}[r]^{g} & Y
    }
\end{equation*}
and its homotopy colimit is called the \textit{homotopy pushout}. 

There is a recursive way to compute homotopy colimits of punctured $n$-cubes. Given a punctured $n$-cube $\mathcal{X}$ for $n\geq2$ and defining the punctured $(n-1)$-cubes
$\mathcal{X}_1(S)=\mathcal{X}(S)$ and $\mathcal{X}_2(S)=\mathcal{X}(S\cup\{n\})$ 
for $S\subsetneq\underline{n-1}$,
we have that (Lemma 5.7.6 \citep{cubicalhomotopy})
$$\hocolim(\mathcal{X})\cong \hocolim\left(\mathcal{X}\left(\underline{n-1}\right)\longleftarrow \hocolim(\mathcal{X}_1)\longrightarrow \hocolim(\mathcal{X}_2)\right).$$

If for all  $S\subsetneq\underline{n}$ the map 
$$\colim_{T\subsetneq S}X_T \longrightarrow X_S$$
is a cofibration, we say the punctured cube is \textit{cofibrant}.  
If $X_1,\dots,X_n$ are CW-complexes such that their intersections are subcomplexes, and we take the punctured cube given by the 
intersections and the inclusions between them, then the punctured cube is cofibrant and 
$\hocolim(\mathcal{X})\simeq \colim(\mathcal{X})$ (Proposition 5.8.25 \citep{cubicalhomotopy}). We will use several times the following 
known fact about homotopy pushouts (see \citep{indcomplcartprod} for a proof).

\begin{lem}\label{homocolimpegado}
Let $X,Y,Z$ be spaces with maps $f\colon Z\longrightarrow X$ and $g\colon Z\longrightarrow Y$ such that both maps are null-homotopic. Then
$$\hocolim\left(\mathcal{S}\right)\simeq X\vee Y\vee \Sigma Z$$
where 
\begin{equation*}
    \xymatrix{
    \mathcal{S}\colon & Y \ar@{<-}[r]^{g} & Z \ar@{->}[r]^{f} & X
    }
\end{equation*}
\end{lem}

As a consequence of the last lemma we obtain the following lemma which we will need when studying the filtration for 
multipartite graphs.

\begin{lem}\label{lemhomotopiagraf}
Let $X_1,\dots,X_k$ be connected simplicial complexes such that the intersection of two or more is either
contractible or empty and there is a graph $G$ of size $k$ and a 
bijection $\gamma\colon\{1,\dots,k\}\longrightarrow E(G)$ such that 
$\bigcap_{i\in S}\gamma(i)\neq\emptyset$ if an only if $\bigcap_{i\in S}X_i\neq\emptyset$ for all non-empty 
$S$ subsets of $\{1,\dots,k\}$. Then 
$\displaystyle X=\bigcup_{i=1}^kX_i$ has the homotopy type of the nerve with the complexes 
$X_i$ attached to the corresponding point in the nerve.
\end{lem}
\begin{proof}
We prove this lemma by induction on $k$. For $k=1,2$ the result is clear. Assume it is true for any $r\leq k$ and take 
$X_1,\dots,X_{k+1}$ simplicial complexes such that the intersection of two or more is 
contractible or empty, $X_i$ is connected for all $i$ and there is a graph $G$ of size $k+1$ and a 
bijection $\gamma\colon\{1,\dots,k+1\}\longrightarrow E(G)$ such that 
$\bigcap_{i\in S}\gamma(i)\neq\emptyset$ if an only if $\bigcap_{i\in S}X_i\neq\emptyset$ for all non-empty 
$S$ subset of $\{1,\dots,k+1\}$. Now, 
take $\mathcal{N}$ the nerve complex of $X_1,\dots,X_{k+1}$. For any $i\in\{1,\dots,k+1\}$, $\mathrm{lk}(i)$ is:
\begin{itemize}
\item[(a)] Empty if the both vertices of the corresponding edge have degree $1$.
\item[(b)] Contractible if one of the vertices of the corresponding  edge has degree $1$ and the other has degree at 
least $2$.
\item[(c)] Homotopy equivalent to $\mathbb{S}^0$ if both of the vertices of the corresponding edge have degree at least 
$2$.
\end{itemize}
By the inductive formula for homotopy colimits of punctured cubes, the homotopy colimit of the intersection diagram associated to $X_1,\dots,X_{k+1}$ is homotopy equivalent to the homotopy pushout of the diagram
\begin{equation*}
\xymatrix{\mathcal{S}\colon\hocolim(\mathcal{S}_2)&\hocolim(\mathcal{S}_1) \ar@{->}[l] \ar@{->}[r]& X_{k+1}},
\end{equation*}
where $S_1$ is the homotopy colimit of the intersection diagram associated to $X_1\cap X_{k+1},\dots,X_k\cap X_{k+1}$, and 
$\mathcal{S}_2$ is the homotopy colimit of the intersection diagram associated to $X_1,\dots,X_k$. 
Now, taking $\mathcal{U}'=\{U_i\cap U_k:\;U_i\cap U_k\neq\emptyset\}$ we have that 
$\mathcal{N(U')}\simeq\hocolim(\mathcal{S}_1)$ and $\mathcal{N(U')}\cong\mathrm{lk}(k+1)$ so we have the same three possibilities as before 
regarding the degrees of the vertices of the corresponding edge. If $G$ is a forest or has a vertex of degree $1$, 
we can assume that of the vertices of the edge corresponding to $X_{k+1}$ at least one has degree one. Therefore:
\begin{itemize}
    \item[(a)] $\hocolim(\mathcal{S}_1)=\emptyset$, then $\hocolim(\mathcal{S})\simeq\hocolim(\mathcal{S}_2)\sqcup X_{k+1}$.
    
    \item[(b)] $\hocolim(\mathcal{S}_1)\simeq*$, then $\hocolim(\mathcal{S})\simeq\hocolim(\mathcal{S}_2)\vee X_{k+1}$.
\end{itemize}
Assume $G$ is such that $\delta(G)\geq2$. We can assume that the edge corresponding to $X_{k+1}$ is not a bridge, therefore we have that 
$\hocolim(\mathcal{S}_1)\simeq\mathbb{S}^0$ and $\hocolim(\mathcal{S})\simeq\hocolim(\mathcal{S}_2)\vee\mathbb{S}^1\vee X_{k+1}$.
\end{proof}

\section{Definition and basic properties}
Let $G$ be a graph, we define its \textit{$d$-forest complex} as the complex
$$\mathcal{F}_d(G)=\{\sigma\subseteq V(G)\colon \:G[\sigma] \mbox{ is a forest with }\Delta(G[\sigma])\leq d\};$$
for $d=\infty$ we take 
$$\mathcal{F}_\infty(G)=\{\sigma\subseteq V(G)\colon \:G[\sigma] \mbox{  is a forest}\}.$$
For $d=0$, $\mathcal{F}_0(G)$ is the independence complex of $G$ and for $d=1$ is also called the \textit{$2$-independence} 
complex of $G$ ---the \textit{$r$-independence} complex of $G$ has as simplices sets $A \subseteq V(G)$ such that every 
connected component of $G[A]$ has at most $r$ vertices. Note 
that if $d+1=\min\{r\colon G\mbox{ is } K_{1,r}\mbox{-free}\}$, then $\mathcal{F}_l(G)=\mathcal{F}_d(G)$ for all $l\geq d$.

Given a graph $G$ let $t_d(G)=\max\{|V(T)|\colon T \mbox{ is an induced forest such that }\Delta(T)\leq d\}$, by definition 
$t_d(G)=\dim(\mathcal{F}_d(G))+1$, 
therefore knowing the homotopy type of $\mathcal{F}_d(G)$ or its homology groups gives us a lower bound for $t_d(G)$. 

\begin{theorem}
For any graph $G$ and all $d$, the pair $(\mathcal{F}_{d+1}(G),\mathcal{F}_d(G))$ is $d$-connected.
\end{theorem}
\begin{proof}
For any $d$, we have that $\mathrm{sk}_i\mathcal{F}_d(G)=\mathrm{sk}_i\mathcal{F}_{d+1}(G)$ for all $i\leq d$ because a forest of order $i+1$ has maximum 
degree at most $i$. Then all the cells in $\mathcal{F}_{d+1}(G)-\mathcal{F}_d(G)$ have dimension greater than $d$ and this implies 
the result (see \citep[Corollary 4.12]{hatcher}).
\end{proof}

\begin{theorem}\label{coneccuello}
For any graph $G$ and $d\ge g(G)-1$, $\mathrm{conn}\left(\mathcal{F}_d(G)\right)\geq g(G)-3$.
\end{theorem}
\begin{proof}
The result follows from:
$$\mathrm{sk}_{g(G)-2}\mathcal{F}_d(G)\cong\mathrm{sk}_{g(G)-2}\Delta^{|V(G)|-1}\simeq\bigvee_{\binom{|V(G)|-1}{g(G)-1}}\mathbb{S}^{g(G)-2}.$$
\end{proof}

The following results are easy to see from the definition of $F_d(G)$.

\begin{prop}
For $d\geq1$,
$$\mathcal{F}_d(K_n)\simeq\bigvee_{\frac{(n-1)(n-2)}{2}}\mathbb{S}^1.$$
\end{prop}

\begin{prop}
For $n\geq3$ and $d\geq2$,
$$\mathcal{F}_d(C_n)\cong\mathbb{S}^{n-2}.$$
\end{prop}

A subset of vertices $\sigma$ is an independent set if all of its subsets of cardinality $2$ are independent. This says that 
in order to be a simplex of the independence complex, a set of vertices only need to have its $1$-skeleton 
contained in the complex. This type of complexes are called \emph{flag complexes}. Now, for 
$\mathcal{F}_1(G)$, its $1$-skeleton is the complete graph of the same order as $G$, therefore it is not a flag complex in general, because 
it is not contractible for all graphs. The following result tells us that it has an analogous property but for the $2$-skeleton, rather than the $1$-skeleton. Sadly this can not be generalized for $\mathcal{F}_d(G)$ with $d\geq2$ as $\mathcal{F}_d(C_{d+3})$ shows.

\begin{prop}
Let $\sigma$ be a subset of $V(G)$ such that all of its subsets of cardinality $3$ are simplices of $\mathcal{F}_1(G)$, then 
$\sigma$ is a simplex of $\mathcal{F}_1(G)$.
\end{prop}
\begin{proof}
If $|\sigma|\leq3$ the result is clear. Now let $\sigma=\{v_0,v_1,v_2,v_3\}$. Then, for $\tau=\{v_0,v_1,v_2\}$, we have that $G_\tau=G[\tau]$ is forest such that 
$\Delta(G_\tau)\leq1$. Now, $v_3$ at most can have one neighbor in $\tau$ and it must be a vertex 
of degree $0$ in $G_\tau$. Therefore $G_\sigma=G[\sigma]$ is a graph such that $\Delta(G_\sigma)\leq1$, which implies it is a forest 
and $\sigma$ is a simplex of $\mathcal{F}_1(G)$.

Assume the result is true for any subset of at most $k\geq 4$ vertices that has its $2$-skeleton in $\mathcal{F}_1(G)$. Let 
$\sigma=\{v_0,\dots,v_k\}$ a subset of $k+1$ vertices such that its $2$-skeleton is in $\mathcal{F}_1(G)$. By induction hypothesis, 
$\tau=\{v_0,\dots,v_{k-1}\}$ is a simplex of $\mathcal{F}_1(G)$, therefore, taking $G_\tau$ as before,
$$G_\tau\cong rK_1 \sqcup M_s$$
with $r,s\geq0$ and $r+2s=k$, where $M_s$ is the graph formed by $s$-disjoint edges. 
By hypothesis, $v_k$ can not be adjacent to a vertex in $M_s$ and only can be adjacent to 
one vertex in $rK_1$. So $\sigma$ induces a graph with maximum degree at most $1$ and therefore $\sigma$ is a simplex of 
$\mathcal{F}_1(G)$.
\end{proof}

A well-known folklore result is that
if $K$ is a simplicial complex such that $\tilde{H}_q(K)\ncong0$, then $f_i(K)\geq f_i\left(\partial\Delta^{q+1}\right)$ and $f_0(K)=q+2$ 
if and only if $K\cong\partial\Delta^{q+1}$. This implies the following.

\begin{prop}
Let $G$ be a graph such that $\tilde{H}_q(\mathcal{F}_d(G))\ncong0$ for some $d$ and $q$, then $G$ has at least $q+2$ different
induced forests of $q+1$ vertices and maximum degree at most $d$.
\end{prop}

\begin{prop}
Let $G$ be a graph of order $q+2$, with $q\geq1$, then:
\begin{enumerate}
    \item If $\tilde{H}_q(\mathcal{F}_q(G))\ncong0$ for $q\neq\infty$, then $G\cong K_{_{1,q+1}}$ or $G\cong C_{q+2}$.
    \item If $\tilde{H}_q(\mathcal{F}_\infty(G))\ncong0$, then $G\cong C_{q+2}$.
\end{enumerate}
\end{prop}
\begin{proof}
For $d=q$ or $d=\infty$, we have that $\mathcal{F}_d(G)\cong\partial\Delta^{q+1}$ and  
for any proper subset of vertices $S$, $\mathcal{F}_d(G[S])$ must be contractible. 
If $\Delta(G)=q+1$, then $G$ can not have cycles because $V(G)-\{v\}$ is a simplex for any vertex and 
$\mathcal{F}_\infty(G)\simeq*$. Take
$v$ a vertex such that $d_G(v)=q+1$, then $\mathcal{F}_q(G)=\mathrm{st}(v)\cup\mathcal{F}_q(G-v)$ and, because 
$\tilde{H}_{q}(\mathcal{F}_q(G-v))\cong0$, using the Mayer-Vietoris sequence we have that 
$\tilde{H}_{q-1}(lk(v))\ncong0$. Therefore $\mathrm{lk}(v)\cong\partial\Delta^{q}$. If $q=1$, then 
$\mathrm{lk}(v)$ is two disjoint vertices from where it follows that $G\cong K_{1,2}$ or $G\cong C_3$. Assume $q\geq2$, 
then $N(v)$ must be an independent set and $G\cong K_{_{1,q+1}}$.

Assume $\Delta(G)\leq q$, then $G$ must have a cycle, otherwise $\mathcal{F}_d(G)\simeq*$ for $d=q$ or 
$d=\infty$. Let $C\leq G$ be an induced cycle. Since any proper subset is a simplex, then $V(C)=V(G)$. 
Therefore $G\cong C_{q+2}$.
\end{proof}

\begin{prop}\label{proppuente}
If $e\in E(G)$ is bridge, then $\mathcal{F}_{\infty}(G)=\mathcal{F}_{\infty}(G-e)$.
\end{prop}

\begin{lem}
If $G=G_1\sqcup G_2$, then for all $d$,
$$\mathcal{F}_d(G)=\mathcal{F}_d(G_1)*\mathcal{F}_d(G_2).$$
\end{lem}

\begin{prop}
If $G=G_1\sqcup\cdots\sqcup G_k$, then for $d\geq0$, $$\mathrm{conn}(\mathcal{F}_d(G))\geq2k-2+\sum_{i=1}^k\mathrm{conn}(\mathcal{F}_d(G_i)).$$
\end{prop}
\begin{proof}
This follows from $\mathcal{F}_d(G)=\mathcal{F}_d(G_1)*\cdots*\mathcal{F}_d(G_k)$
\end{proof}

\begin{lem}\label{lemacyvert}
If $v$ is a vertex such that no cycle of $G$ contains it, then 
$\mathcal{F}_\infty(G)\simeq*$.
\end{lem}
\begin{proof}
Because $v$ does not belong to a cycle, then $\mathcal{F}_\infty(G)=\{\{v\}\}*\mathcal{F}_\infty(G-v)$.
\end{proof}

\begin{cor}
If $\delta(G)\leq1$, then $\mathcal{F}_\infty(G)\simeq*$.
\end{cor}

\begin{cor}\label{corgrad2}
If $G$ has a vertex $v$ such that $N_G(v)=\{v_1,v_2\}$, then 
$\mathcal{F}_\infty(G)\simeq\Sigma\mathrm{lk}_{_{\mathcal{F}_\infty(G)}}(v_i)$ for $i=1,2$.
\end{cor}
\begin{proof}
Because $N_G(v)=\{v_1,v_2\}$, then $d_{G-v_i}(v)=1$ and  
$\mathcal{F}_\infty(G-v_i)\simeq*$. Now $\mathcal{F}_\infty(G)\simeq\hocolim(\mathcal{S})$ with  
$\mathcal{S}\colon\mathcal{F}_\infty(G-v_i)\longhookleftarrow\mathrm{lk}_{_{\mathcal{F}_\infty(G)}}(v_i)\longhookrightarrow\mathrm{st}_{_{\mathcal{F}_\infty(G)}}(v_i)$, 
by Lemma \ref{homocolimpegado} we obtain the result.
\end{proof}

The previous corollary implies that if the minimum degree of a graph $G$ is at most two, then $\mathcal{F}_\infty(G)$ has 
the homotopy type of a suspension.

\begin{lem}\label{lemlinkvertrian}
Let $G$ be a graph that is the union of three graphs $G_1,G_2,G_0$ such that: 
\begin{itemize}
    \item $G_0\cong K_3$
    \item $V(G_0)=\{v,v_1,v_2\}$
    \item $V(G_1)\cap V(G_0)=\{v_1\}$, $V(G_2)\cap V(G_0)=\{v_2\}$ and $V(G_1)\cap V(G_2)=\emptyset$
\end{itemize}
Then, $\mathrm{lk}_{\mathcal{F}_\infty(G)}(v)\simeq\hocolim(\mathcal{S})$ with $\mathcal{S}$ the diagram:
\begin{equation*}
    \xymatrix{\mathcal{F}_\infty(G_1)*\mathcal{F}_\infty(G_2-v_2)\ar@{<-^)}[r] & \mathcal{F}_\infty(G_1-v_1)*\mathcal{F}_\infty(G_2-v_2)\ar@{^(->}[r]& \mathcal{F}_\infty(G_1-v_1)*\mathcal{F}_\infty(G_2)}
\end{equation*}
\end{lem}
\begin{proof}
Because 
$$\mathrm{lk}_{\mathcal{F}_\infty(G)}(v)=(\mathcal{F}_\infty(G_1)*\mathcal{F}_\infty(G_2-v_2))\cup(\mathcal{F}_\infty(G_1-v_1)*\mathcal{F}_\infty(G_2))$$
and
$$(\mathcal{F}_\infty(G_1)*\mathcal{F}_\infty(G_2-v_2))\cap(\mathcal{F}_\infty(G_1-v_1)*\mathcal{F}_\infty(G_2))=\mathcal{F}_\infty(G_1-v_1)*\mathcal{F}_\infty(G_2-v_2)$$
we have that 
$$\mathrm{lk}_{\mathcal{F}_\infty(G)}(v)=\colim(\mathcal{S})\simeq\hocolim(\mathcal{S})$$
\end{proof}

\section{Homotopy type calculations}
In this section we will calculate the homotopy type for the forest complexes for various families of graphs, 
for every step in the filtration in most cases; and in some cases, like the cactus graphs, only for the last complex of the filtration.
\subsection{Some Graph families}
\subsubsection{Paths and cycles}
The homotopy type of all $r$-independence complexes of paths was calculated by Salvetti \citep{salvetti2018} using discrete
Morse theory. Here we give a different proof for $\mathcal{F}_1$ using homotopy pushouts, which also shows that 
the inclusion $\mathcal{F}_1(P_{4r+3})\longhookrightarrow\mathcal{F}_1(P_{4(r+1)})$ is a homotopy equivalence. This will allow us 
to calculate the homotopy type of $\mathcal{F}_1(C_n)$ avoiding  discrete Morse theory, which was the tool used in \citep{singhhigher}.
\begin{prop}\citep{salvetti2018}
$$\mathcal{F}_1(P_n)\simeq\left\lbrace\begin{array}{cc}
    \mathbb{S}^{2r-1} &  \mbox{ if }n=4r\\
    \ast & \mbox{ if }n=4r+1\mbox{ or }n=4r+2\\
    \mathbb{S}^{2r+1} & \mbox{ if }n=4r+3
\end{array}
\right.$$
\end{prop}
\begin{proof}
For $r=0$, it is clear that $\mathcal{F}_1(P_1)\simeq*\simeq\mathcal{F}_1(P_2)$. For $P_3$, 
$\mathcal{F}_1(P_3)\cong K_3$. For $\mathcal{F}_1(P_4)$ 
$$\mathrm{lk}(v_4)=\mathcal{F}_1(P_2)\cup\{v_3\}*\mathcal{F}_1(P_1)\simeq*$$
therefore the inclusion $i\colon \mathcal{F}_1(P_3)\longhookrightarrow\mathcal{F}_1(P_4)$ is a homotopy equivalence.

Next, we will prove that $\mathcal{F}_1(P_{4r+1})\simeq\mathcal{F}_1(P_{4r+2})\simeq*$  for all $r\geq1$.

Assume that it is true for any $1\leq r\leq k$. For $\mathcal{F}_1(P_{4(k+1)})$, by induction hypothesis,
$$\mathrm{lk}(v_{4(k+1)})=\mathcal{F}_1(P_{4k+2})\cup\{v_{4k+3}\}*\mathcal{F}_1(P_{4k+1})\simeq*$$
therefore the inclusion $\mathcal{F}_1(P_{4k+3})\longhookrightarrow\mathcal{F}_1(P_{4(k+1)})$ is 
a homotopy equivalence.

Now, for $\mathcal{F}_1(P_{4(k+1)+1})$ we have
$$\mathrm{lk}(v_{4(k+1)+1})=\mathcal{F}_1(P_{4k+3})\cup\{v_{4(k+1)}\}*\mathcal{F}_1(P_{4k+2}).$$
Setting $X=\mathcal{F}_1(P_{4k+2})$ and $Y=\{v_{4k+3}\}*\mathcal{F}_1(P_{4k+2})$, we have by induction hypothesis that 
$$X\cap Y=\mathcal{F}_1(P_{4k+2})\simeq*$$
therefore $\mathcal{F}_1(P_{4k+3})\longhookrightarrow\mathrm{lk}(v_{4(k+1)+1})$ is a homotopy equivalence.
\begin{equation*}
    \xymatrix{
    \mathcal{F}_1(P_{4k+3}) \ar@{->}[r]_{\simeq} \ar@{->}[d] \ar@{->}@/^{5mm}/[rr]^{\simeq} & \mathrm{lk}(v_{4(k+1)+1}) \ar@{->}[r] \ar@{->}[d] & \mathcal{F}_1(P_{4(k+1)}) \ar@{->}[d] \\
    \mathrm{st}(v_{4(k+1)+1}) \ar@{->}[r] & \mathrm{st}(v_{4(k+1)+1}) \ar@{->}[r]& \mathcal{F}_1(P_{4(k+1)+1})
    }
\end{equation*}

$$\mathcal{F}_1(P_{4(k+1)+1})\simeq\mathrm{st}(v_{4(k+1)+1})\simeq*$$

For $\mathcal{F}_1(P_{4(k+1)+2})$:
$$\mathrm{lk}(v_{4(k+1)+2})=\mathcal{F}_1(P_{4(k+1)})\cup\{v_{4(k+1)+1}\}*\mathcal{F}_1(P_{4k+3});$$
because $\mathcal{F}_1(P_{4k+3})\longhookrightarrow\mathcal{F}_1(P_{4(k+1)})$ is an homotopy equivalence, we have 
that $\mathrm{lk}(v_{4(k+1)+2})\simeq*$ and therefore 
$$\mathcal{F}_1(P_{4(k+1)+2})\simeq\mathcal{F}_1(P_{4(k+1)+1})\simeq*.$$

We have that $\mathcal{F}_1(P_{4(k+1)})\simeq\mathcal{F}_1(P_{4k+3})$; now for this last complex:
$$\mathrm{lk}(v_{4k+3})=\mathcal{F}_1(P_{4k+1})\cup\{v_{4k+2}\}*\mathcal{F}_1(P_{4k}),$$
where $\mathcal{F}_1(P_{4k+1})\simeq*$, therefore 
$$\mathrm{lk}(v_{4k+3})\simeq\Sigma \mathcal{F}_1(P_{4k}).$$
Since $\mathcal{F}_1(P_{4k+2})\simeq*$, we have that $\mathcal{F}_1(P_{4k+3})\simeq\Sigma^2\mathcal{F}_1(P_{4k})$
and 
$$\mathcal{F}_1(P_{4(k+1)})\simeq\Sigma^2\mathcal{F}_1(P_{4k})\simeq\Sigma^2\mathbb{S}^{2k-1}\simeq\mathbb{S}^{2k+1}.$$
Doing the exact same argument we can see that $\mathcal{F}_1(P_{4(k+1)+3})\simeq\Sigma^2\mathcal{F}_1(P_{4(k+1)})$
and therefore 
$$\mathcal{F}_1(P_{4(k+1)+3})\simeq\Sigma^2\mathbb{S}^{2k+1}\simeq\mathbb{S}^{2k+3}.$$
\end{proof}

In the proof of the last proposition we saw that the inclusion $\mathcal{F}_1(P_{4k+3})\longhookrightarrow\mathcal{F}_1(P_{4(k+1)})$ 
obtained by erasing the last (or the first) vertex is an homotopy equivalence. We will use this fact in the following corollary.

\begin{cor}\citep{singhhigher}\label{corciclos1}
$$\mathcal{F}_1(C_n)\simeq\left\lbrace\begin{array}{cc}
    \displaystyle\bigvee_{3}\mathbb{S}^{2r-1} & \mbox{ if } n=4r \\
    \mathbb{S}^{2r-1} & \mbox{ if } n=4r+1 \\
    \mathbb{S}^{2r} & \mbox{ if } n=4r+2 \\
    \mathbb{S}^{2r+1} & \mbox{ if } n=4r+3
\end{array}\right.$$
\end{cor}
\begin{proof}
For $n=3,4$, the only possible simplices are a vertex or pair of vertices, any set with more vertices will have a $3$-path or a cycle. 
Therefore $\mathcal{F}_1(C_3)\cong K_3$ and $\mathcal{F}_1(C_4)\cong K_4$. For $n=5$, taking $v_1,v_2,v_3,v_4,v_5$ the vertices of 
the cycle with edges $v_iv_{i+1}$, the facets of $\mathcal{F}_1(C_5)$ are $\sigma_i=\{v_i,v_{i+2},v_{i+3}\}$. The edge $v_{i+2}v_{v_i+3}$ 
only is contained in $\sigma_i$, so we can collapse it for all $i$. Therefore $\mathcal{F}_1(C_5)\simeq\mathcal{F}_0(C_5)\cong\mathbb{S}^1$.

Assume $n\geq6$ and let $v_1,\dots,v_n$ be the vertices of the cycle. Then $lk(v_n)=K_1\cup K_2\cup K_3$ where
$$K_1=\mathcal{F}_1(C_n-v_n-v_2-v_{n-1})\cong C(\mathcal{F}_1(P_{n-4}))$$
$$K_2=\mathcal{F}_1(C_n-v_n-v_1-v_{n-2})\cong C(\mathcal{F}_1(P_{n-4}))$$
$$K_3=\mathcal{F}_1(C_n-v_n-v_1-v_{n-1})\cong \mathcal{F}_1(P_{n-3})$$
Now
$$K_1\cap K_2\cap K_3=K_1\cap K_2=\mathcal{F}_1(C_n-v_n-v_1-v_2-v_{n-1}-v_{n-2})\cong \mathcal{F}_1(P_{n-5})$$
$$K_1\cap K_3=\mathcal{F}_1(C_n-v_n-v_1-v_2-v_{n-1})\cong \mathcal{F}_1(P_{n-4})$$
$$K_2\cap K_3=\mathcal{F}_1(C_n-v_n-v_1-v_{n-1}-v_{n-2})\cong \mathcal{F}_1(P_{n-4})$$
$$K_1\cup K_2\simeq\Sigma\mathcal{F}_1(P_{n-5})$$

If $n=4r$, $K_1\cap K_2\cong\mathcal{F}_1(P_{4(r-2)+3})$, $K_3\simeq*$ and 
$K_1\cap K_3\cong\mathcal{F}_1(P_{4(r-1)})\cong K_2\cap K_3$. 
By the observation before the corollary, the inclusion 
$K_1\cap K_2\cap K_3\longhookrightarrow K_1\cap K_3$ is a homotopy equivalence. Therefore 
$(K_1\cup K_2)\cap K_3\simeq K_2\cap K_3$ and 
$$\mathrm{lk}(v_n)\simeq\bigvee_{2}\mathbb{S}^{2r-2},$$
Since
$$\mathcal{F}_1(C_{4r}-v_n)\simeq\mathbb{S}^{2r-1},$$
we obtain the result.

If $n=4r+1$, $K_1\cap K_3\simeq K_2\cap K_3\cong\mathcal{F}_1(P_{4(r-1)+1})\simeq*$ and 
$K_2\cup K_3\simeq K_3$. Because $K_1\cap K_2\cap K_3=K_1\cap K_2$, we have that 
$$(K_2\cup K_3)\cap K_1\simeq K_1\cap K_3\simeq*$$
and 
$$K_1\cup K_2\cup K_3\simeq K_2\cup K_3\simeq K_3\cong\mathcal{F}_1(P_{4(r-1)+2})\simeq*.$$
Therefore $\mathcal{F}_1(C_{4r+1})\simeq\mathcal{F}_1(P_{4r})\simeq\mathbb{S}^{2r-1}$.

For $n=4r+2$ and $n=4r+3$, $\mathcal{F}_1(C_n-v_n)\simeq*$, therefore $\mathcal{F}_1(C_n)\simeq\Sigma\mathrm{lk}(v_n)$.
If $n=4r+2$, $K_1\cap K_2\cong\mathcal{F}_1(P_{4(r-1)+1})\simeq*$ and 
$K_1\cap K_3,K_2\cap K_3\cong\mathcal{F}_1(P_{4(r-1)+2})\simeq*$. Then 
$K_1\cup K_2\simeq*$ and $(K_1\cup K_2)\cap K_3\simeq*$. From this we have that $\mathrm{lk}(v_n)\simeq K_3$, therefore
$$\mathcal{F}_1(C_{4r+2})\simeq\Sigma\mathcal{F}_1(P_{4(r-1)+3})\simeq\mathbb{S}^{2r}.$$

If $n=4r+3$, $K_2\cap K_3\cong\mathcal{F}_1(P_{4(r-1)+3})$ and the inclusion $K_2\cap K_3\longhookrightarrow K_3$ is a homotopy 
equivalence, therefore $K_2\cup K_3\simeq*$. From this $\mathrm{lk}(v_n)\simeq\Sigma(K_1\cap(K_2\cup K_3))$. Since 
$K_1\cap K_2\cap K_3=K_1\cap K_2$, we have that $K_1\cap(K_2\cup K_3)\simeq K_1\cap K_3$ and 
$$\mathcal{F}_1(C_{4r+3})\simeq\Sigma^2\mathcal{F}_1(P_{4(r-1)+3})\simeq\mathbb{S}^{2r+1}.$$
\end{proof}

\begin{prop}\label{propciclocuerda}
$$\mathcal{F}_{\infty}(C_n+e)\cong\mathbb{S}^{n-3}$$
\end{prop}
\begin{proof}
Assume the vertices of $G=C_n+e$ are labeled $v,w_1,\dots,w_r,u,w_{r+1},\dots,w_{r+k}$ with $e=vu$ (Figure \ref{c4c4}). 
Because $\mathcal{F}_\infty(G-v)\simeq*$, we have that $\mathcal{F}_\infty(G)\simeq\Sigma lk(v)$. Now, $lk(v)$ is formed by the subsets of 
$V(G-v)$ such that together with $v$ they do not induce a cycle, therefore the facets are 
$$\sigma_0=[w_1,\dots,w_r,w_{r+1},\dots,w_{r+k}]$$ 
and 
$$\sigma_{ij}=[w_1,\dots,\hat{w}_i,\dots,w_r,u,w_{r+1},\dots,,\hat{w}_{r+j},\dots,w_{r+k}]$$
for $1\leq i\leq r$, $1\leq j\leq k$.
If we call $K$ the complex form by $\sigma_0$ and its subsets, and $L$ the complex with facets the simplices 
$\sigma_{ij}$, we get that $lk(v)=K\cup L$ and both of these complexes are contractible, therefore 
$\mathrm{lk}(v)\simeq\Sigma K\cap L$.

Now, taking $X$ the complex with facets $[w_{r+1},\dots,,\hat{w}_{r+j},\dots,w_{r+k}]$  and 
$Y$ the complex with facets $[w_1,\dots,\hat{w}_i,\dots,w_r]$, we have that 
$K\cap L\cong X*Y$. Because $X\cong\mathbb{S}^{k-2}$ and 
$Y\cong\mathbb{S}^{r-2}$, we have that
$K\cap L\cong\mathbb{S}^{k-2}*\mathbb{S}^{r-2}\cong\mathbb{S}^{r+k-3}$ and, because $r+k=n-2$, 
$\mathcal{F}_\infty(G)\simeq\mathbb{S}^{n-3}$.
\end{proof}

\begin{figure}
\centering
\begin{tikzpicture}[line cap=round,line join=round,>=triangle 45,x=1cm,y=1cm]
\clip(-2.5,-1.5) rectangle (2.5,1.5);
\draw (0,1)-- (0,-1);
\draw (0,1)-- (2,1);
\draw (0,1)-- (-2,1);
\draw (0,-1)-- (2,-1);
\draw (0,-1)-- (-2,-1);
\draw[dashed] (2,1)-- (2,-1);
\draw[dashed] (-2,1)-- (-2,-1);
\begin{scriptsize}
\fill [color=black] (-2,1) circle (1.5pt);
\fill [color=black] (0,1) circle (1.5pt);
\fill [color=black] (2,1) circle (1.5pt);
\fill [color=black] (-2,-1) circle (1.5pt);
\fill [color=black] (0,-1) circle (1.5pt);
\fill [color=black] (2,-1) circle (1.5pt);
\draw[color=black] (0,1.2) node {$v$};
\draw[color=black] (0,-1.2) node {$u$};
\draw[color=black] (2,1.2) node {$w_1$};
\draw[color=black] (2,-1.2) node {$w_r$};
\draw[color=black] (-2,-1.2) node {$w_{r+1}$};
\draw[color=black] (-2,1.2) node {$w_{r+k}$};
\end{scriptsize}
\end{tikzpicture}
\caption{$C_n+e$}\label{c4c4}
\end{figure}

\subsubsection{Double stars}

Let $St_{r,s}$ be the \emph{double star} with $V(St_{r,s})=\{u_0,u_1,\dots,u_r,v_0,v_1,\dots,v_s\}$ and 
$E(St_{r,s})=\{u_iu_0\colon \;i>0\}\cup\{v_iv_0\colon \;i>0\}\cup\{u_0v_0\}$
\begin{prop}
$$\mathcal{F}_1(St_{r,s})\simeq\mathbb{S}^1$$
and for $2\leq d<\infty$
$$\mathcal{F}_d(St_{r,s})\simeq\bigvee_{\binom{r-1}{d-1}\binom{s-1}{d-1}}\mathbb{S}^{2d-1}$$
\end{prop}
\begin{proof}
For $\mathcal{F}_1(St_{r,s})$, the link of $u_0$ has as facets $\sigma_i=\{u_i,v_1,\dots,v_s\}$ for all $i$ and $\{v_0\}$, therefore
$$\mathrm{lk}(u_0)\simeq\mathbb{S}^0.$$
Since $\mathcal{F}_1(St_{r,s}-u_0)\simeq*$, we have that $\mathcal{F}_1(St_{r,s})\simeq\Sigma\mathrm{lk}(u_0)\simeq\mathbb{S}^1$.

For $d\geq2$, if $r\leq d-1$ or $s\leq d-1$, then $\mathcal{F}_d(St_{r,s})\simeq*$, because the set $\{u_1,\dots,u_r\}$ or the set 
$\{v_1,\dots,v_s\}$ would be contained in all facets.
Assume $r,s\geq d$. 
The facets of $\mathcal{F}_d(St_{r,s})$, besides 
$X=\{u_1,\dots,u_r,v_1\dots,v_s\}$, are of $3$ types:
\begin{enumerate}
    \item $\alpha_S=S\cup\{u_0,v_1,\dots,v_s\}$, where $S\subseteq\{u_1,\dots,u_r\}$ and $|S|=d$.
    \item $\beta_S=S\cup\{v_0,u_1,\dots,u_r\}$, where $S\subseteq\{v_1,\dots,v_s\}$ and $|S|=d$.
    \item $\sigma_{_{S_1,S_2}}=\{u_0,v_0\}\cup S_1\cup S_2$, where $S_1\subseteq\{u_1,\dots,u_r\}$, $S_2\subseteq\{v_1,\dots,v_s\}$ and 
    $|S_1|=|S_2|=d-1$.
\end{enumerate}
Take $\tau=\mathcal{P}(X)$, $\alpha$ the complex generated by $\{\alpha_S\}$, 
$\beta$ the complex generated by $\{\beta_S\}$ and $\sigma$ the complex generated by the $\{\sigma_{_{S_1,S_2}}\}$, 
$\mathcal{F}_d(St_{r,s})=\alpha\cup\beta\cup\sigma\cup\tau$.
Now, these four complexes are contractible and so are $\alpha\cap\sigma,\beta\cap\sigma,\alpha\cap\tau,\beta\cap\tau$. 
Also
$$\alpha\cap\beta\cap\sigma\cap\tau=\alpha\cap\sigma\cap\tau=\beta\cap\sigma\cap\tau=\alpha\cap\beta\cap\sigma=\sigma\cap\tau\cong 
\mathrm{sk}_{d-2}\Delta^{r-1}*\mathrm{sk}_{d-2}\Delta^{s-1}$$
and $\alpha\cap\beta\cap\tau=\alpha\cap\beta$. We compute the homotopy colimit of the punctured $4$-cube given by this union using the recursive formula given in the preliminaries. This what the formula gives applied to the top and bottom of the $4$-cube:

\begin{equation*}
\xymatrix{
\alpha\cap\beta\cap\sigma\cap\tau \ar@{->}[rr] \ar@{->}[dr] \ar@{->}[dd] & & \alpha\cap\beta\cap\sigma \ar@{-}[d] \ar@{->}[rd] & & \\
 & \alpha\cap\sigma\cap\tau \ar@{->}[rr] \ar@{->}[dd] & \ar@{->}[d]& \sigma\cap\tau \ar@{->}[r] \ar@{->}[dd] & \alpha\cap\sigma \ar@{->}[dd] \\
\beta\cap\sigma\cap\tau \ar@{-}[r] \ar@{->}[dr]^{\cong} & \ar@{->}[r] & \beta\cap\sigma \ar@{->}[dr]^{\simeq} &  & \\
  & \sigma\cap\tau \ar@{->}[rr] &  & \ast \ar@{->}[r] & \Sigma(\sigma\cap\tau)
}
\end{equation*}

\begin{equation*}
\xymatrix{
\alpha\cap\beta\cap\tau \ar@{->}[rr]^{\cong} \ar@{->}[dr] \ar@{->}[dd] & & \alpha\cap\beta \ar@{-}[d] \ar@{->}[rd] & & \\
 & \alpha\cap\tau \ar@{->}[rr]^{\simeq} \ar@{->}[dd] & \ar@{->}[d] & \ast \ar@{->}[r] \ar@{->}[dd] & \alpha \ar@{->}[dd] \\
\beta\cap\tau \ar@{-}[r] \ar@{->}[dr] & \ar@{->}[r] & \beta \ar@{->}[dr] &  & \\
  & \tau \ar@{->}[rr] &  & \ast \ar@{->}[r] & \ast
}
\end{equation*}
We find that the complex has the homotopy type of the following homotopy pushout:
$$\mathcal{S}\colon \ast\longleftarrow\Sigma(\sigma\cap\tau)\longrightarrow\tau$$
where $\tau$ is contractible. Therefore
$$\hocolim(\mathcal{S})\simeq\Sigma^2(\sigma\cap\tau)\simeq\bigvee_{\binom{r-1}{d-1}\binom{s-1}{d-1}}\mathbb{S}^{2d-1}.$$
\end{proof}

\subsubsection{Cactus graphs}

For any graph $G$, we take the block graph $B(G)$ in which the vertices are the blocks of $G$ and the cut-vertices of $G$, where 
$vB$ is an edge if $v$ is a vertex of $B$. If $G$ is connected, then $B(G)$ is a tree.

A graph $G$ is a \textit{cactus graph} if all of its blocks are isomorphic to a cycle or to $K_2$. We will say that 
a block is \textit{saturated} if all of its vertices are cut vertices and $sb(G)$ is the number of saturated blocks. A 
vertex $v$ is \textit{saturated} if it is shared by two or more saturated blocks, with $sv(G)$ the number of saturated vertices.

\begin{lem}\label{lemblocvertsatu}
Let $G$ be a cactus graph such that $sb(G)\geq1$, then there is a saturated block $B$ such that either it does not have saturated vertices, or:
\begin{itemize}
    \item[(i)] It has only one saturated vertex $v$.
    \item[(ii)] The connected component of $B(G)-v$ which contains $B$ does not have 
    any other saturated block.
\end{itemize}
\end{lem}
\begin{proof}
If there are no saturated vertices, there is nothing to prove. Assume $sv(G)\geq1$. If there is a saturated block 
without a saturated vertex, again there is nothing to prove. Assume all saturated blocks have at least one saturated vertex. 

Let $V_1$ be the set of all saturated blocks of $G$ and $V_2$ the set of all saturated vertices. In the subgraph  
$T=B(G)[V_1\cup V_2]$ all the leaves are blocks, because each saturated vertex is in at least two saturated blocks, therefore 
$d_T(v)\geq2$ for all $v$ in $V_2$. We take $L\subseteq V_1$ the set of all the leaves of $T$ and let $(B_1,B_2)$ be a pair 
in $L\times L$ such that 
$$d_{_{B(G)}}(B_1,B_2)=\max\{d_{_{B(G)}}(X,Y)\colon \;(X,Y)\in L\times L\}$$
Take $v_1$ the only saturated vertex in $B_1$ and $v_2$ the only saturated vertex in $B_2$. We claim that the only 
$B_1B_2$-path in $B(G)$ contains both $v_1$ and $v_2$. If not, then $B_1$ and $B_2$ are in different connected components of $T$ and, 
assuming $v_1$ is not in the $B_1B_2$-path, any leaf $B'$ in the same component of $B_1$ is such that 
$d(B',B_2)>d(B_1,B_2)$. Therefore $v_1$ and $v_2$ are in the only $B_1B_2$-path.

If in $B(G)-v_1$ there are saturated blocks in the same component than $B_1$, the distance between these and $B_2$ is larger that the 
distance between $B_1$ and $B_2$, which can not happen. Therefore $B_1$ and $v_1$ are as wanted. 
\end{proof}

\begin{lem}\label{homotdual}
Let $G$ be a cactus graph such that  all of its blocks are cycles and such that it does not have saturated blocks, then 
$$\mathcal{F}_\infty^*(G)\simeq\mathbb{S}^{b(G)-2}.$$
\end{lem}
\begin{proof}
Let $B_0,\dots,B_k$ be the blocks of $G$. If $k=0$, then $\mathcal{F}_\infty^*(G)=\emptyset=\mathbb{S}^{-1}$. Assume,
$k\geq1$. We take $X_i=V(G)-V(B_i)$ for all $i$, this are the facets of $\mathcal{F}_\infty^*(G)$ and we have that
$$\bigcap_{i=0}^kX_i=\emptyset$$
$$\bigcap_{i\in S}X_i\neq\emptyset, \;\; \forall S\subsetneq [k]$$
Then, its nerve is isomorphic to $\partial\Delta^k\cong\mathbb{S}^{k-1}$. Therefore,
$\mathcal{F}_\infty^*(G)\simeq\mathbb{S}^{b(G)-1}$.
\end{proof}

\begin{lem}\label{cactussimpcon}
Let $G$ be a cactus graph different from $K_3$, then $\mathcal{F}_\infty(G)$ is simply connected.
\end{lem}
\begin{proof}
If $G$ has only one block and $G$ is not $K_3$, $G$ must be a single vertex, $K_2$ or a cycle with at least $4$ vertices, thus 
$\mathcal{F}_\infty(G)$ is contractible or a sphere of dimension at least $2$. Assume $G$ has $k\geq2$ blocks.
For each block that is not isomorphic to $K_2$ we can erase one edge to we obtain $T$, a spanning tree of 
$G$ and $\mathcal{F}_\infty(G)$. Taking the free group $H_T$ with $E(G)\cup E(G^c)$ as generators and with the relations  
\begin{itemize}
    \item $uv=1$ for all the edges of $T$
    \item $(uv)(vw)=uw$ if $\{u,v,w\}$ is a simplex of $\mathcal{F}_\infty(G)$
\end{itemize}
we have that $H_T\cong\pi_1\left(\mathcal{F}_\infty(G)\right)$ (see \citep[Theorem 7.34]{rotmantop}). Take $uv\in E(G)\cup E(G^c)-E(T)$. 

If $u,v$ are in the same block, this block must be a cycle. If the cycle has $4$ or more vertices, there is a $uv$-path 
$uw_1w_2\cdots w_rv$  in $T$. Now, $\{u,w_1,v\},\{w_1.w_2,v\},\dots,\{w_{r-1},w_r,v\}$ are simplicies of $\mathcal{F}_\infty(G)$, then
$uv=w_1v=w_2v=\cdots=w_lv=1$. If the cycle is $uvw$, because there are $k\geq2$ blocks, one of the vertices must be a cut 
vertex:
\begin{itemize}
    \item If $u$ is a cut vertex, $u$ has a neighbor $x$ in another block such that $ux$ is in $T$. Then 
    $\{u,v,x\}$ is a simplex of $\mathcal{F}_\infty(G)$ and $uv=xv$. Now, $\{v,w,x\}$ and $\{u,w,x\}$ are simplices, thus 
    $xv=xw=uw=1$. The case in which $v$ is a cut vertex is analogous.
    \item If $w$ is a cut vertex, $w$ has a neighbor $x$ in another block such that $wx$ is in $T$. Then 
    $\{u,v,x\}$, $\{u,w,x\}$ and $\{v,w,x\}$ are simplices. Therefore $uv=(ux)(xv)=((uw)(wx))((xw)(wv))=1$.
\end{itemize}

If $u,v$ are in different blocks, then there are cut vertices $w_1,\dots,w_r$, with $r\geq1$, such that: they are on the only $uv$-path in 
$T$; $w_j$ it is not in the only $uw_i$-path in $T$ for any $j>i$; and there are no more cut vertices in the only $uv$-path in $T$. 
Then $\{u,w_1,v\},\{w_1,w_2,v\},\dots,\{w_{r-1},w_r,v\}$ are simplices and $uv=w_1v=w_2v=\cdots=w_rv=1$.

Therefore $\pi_1(\mathcal{F}_\infty(G))\cong H_T\cong0$.
\end{proof}

\begin{cor}\label{cornoblocsat}
Let $G$ a cactus graph such all of its blocks are cycles and does not have saturated blocks, then 
$$\mathcal{F}_\infty(G)\simeq\mathbb{S}^{n-b(G)-1}.$$
\end{cor}
\begin{proof}
If $b(G)=1$, then $G$ is a cycle and $\mathcal{F}_\infty(G)\cong\mathbb{S}^{n-2}$. Assume $b(G)\geq2$, then, by Lemma \ref{cactussimpcon}, 
$\mathcal{F}_\infty(G)$ is simply connected and , by Lemma  \ref{homotdual}, $\mathcal{F}_\infty^*(G)\simeq\mathbb{S}^{b(G)-1}$. 
Therefore, by Theorem \ref{dualidadalexander}, $\mathcal{F}_\infty(G)$ is a simply connected complex such that its only nontrivial 
reduced homology group is in dimension $q=n-b(G)-1$, which is isomorphic to $\mathbb{Z}$. By Theorem \ref{cwsimplcon}, 
$\mathcal{F}_\infty(G)$ is homotopy equivalent to a sphere of the desired dimension.
\end{proof}

\begin{theorem}
If $G$ is a cactus graph then $\mathcal{F}_\infty(G)$ is either contractible or homotopy equivalent to a sphere of dimension at least 
$n-b(G)-1$.
\end{theorem}
\begin{proof}
If $\delta(G)=1$, then $\mathcal{F}_\infty(G)\simeq*$. Assume $\delta(G)=2$. If there is a cut vertex of degree $2$, by 
Lemma \ref{lemacyvert},
$\mathcal{F}_\infty(G)\simeq*$. Assume there is no cut vertex of degree $2$. If $G$ has a bridge $e$, then $G-e=G_1+G_2$ and, 
by Proposition \ref{proppuente}, 
$\mathcal{F}_\infty(G)=\mathcal{F}_\infty(G_1)*\mathcal{F}_\infty(G_2)$. If $G$ has more bridges, then we continue this 
process until we get that $\mathcal{F}_\infty(G)=\mathcal{F}_\infty(H_1)*\cdots*\mathcal{F}_\infty(H_{r+1})$, where $r$ 
is the number of bridges and each $H_i$ is a cactus graph such 
that every block is a cycle. So, if every $\mathcal{F}_\infty(H_i)$ has $n_i$ vertices, is not contractible and is homotopy equivalent to a 
sphere of dimension at least $n_i-b(H_i)-1$, $\mathcal{F}_\infty(G)$ will be homotopy equivalent to a sphere of dimension at least 
$n-b(G)+r-1>n-b(G)-1$. Therefore we only need to prove the result for cactus graphs which do not have blocks isomorphic to $K_2$.

If $G$ does not have saturated blocks, by Corollary \ref{cornoblocsat}, 
$$\mathcal{F}_\infty(G)\simeq\mathbb{S}^{n-b(G)-1}.$$
So assume $sb(G)\geq1$, which implies that $b(G)\geq4$. 
Now, we prove the result by induction on $sv(G)$. If $sv(G)=0$, then take $B_0$ a saturated block of 
$G$ and $B_1,\dots,B_k$ the remaining blocks. Let $X_i=V(G)-V(B_i)$, then $X_0,X_1,\dots,X_k$ are the facets of 
$\mathcal{F}_\infty^*(G)$. Because $B_0$ is saturated, 
$$\bigcap_{i=1}^kX_i=\emptyset.$$
Let $S\subseteq [k]-\{0\}$ such that 
$$\sigma=\bigcap_{i\in S}X_i\neq\emptyset.$$
Then there is $0<j\leq k$ such that $j\notin S$ and $V(B_j)\cap\sigma\neq\emptyset$, with $B_j$ a non-saturated block or 
a saturated block (which can not share vertices with $B_0$).
Then there is a vertex $v$ in $V(B_j)$ such that $v$ is not a vertex of the blocks with index in 
$S$ nor is a vertex of $B_0$, therefore $v\in X_0$, $v\in\sigma$ and 
$X_0\cap\sigma\neq\emptyset$. From this we get that the nerve is a cone with apex vertex $X_0$ and $\mathcal{F}_\infty^*(G)\simeq*$. 
Then, by Lemma \ref{cactussimpcon} and Theorem \ref{dualidadalexander}, 
$\mathcal{F}_\infty(G)$ is simply connected and all of its reduced homology groups are 
trivial. Therefore, by Theorem \ref{whiteheadhomologia}, $\mathcal{F}_\infty(G)$ is contractible.
This argument only used that there is an isolated saturated block-- a saturated block which does not 
have saturated vertices. Therefore we can assume that there is no isolated saturated block.

Assume the result is true for $sv(G)\leq k$ and let $G$ be a cactus graph with $sv(G)=k+1$ and with all of its blocks isomorphic to cycles. 
By Lemma \ref{lemblocvertsatu} there is $B_0$ a saturated block such that only one of its vertices is a saturated vertex, say $v$, and in 
the connected component of $B(G)-v$ which contains $B_0$ there are no more saturated blocks. We call $G_1$ the subgraph formed by the blocks 
in this connected component, and $G_2$ the subgraph induced by the remaining blocks. Then $G=G_1\cup G_2$ and $G_1\cap G_2\cong K_1$. Now 
$$\mathrm{lk}_{\mathcal{F}_\infty(G)}(v)=\mathrm{lk}_{\mathcal{F}_\infty(G_1)}(v)*\mathrm{lk}_{\mathcal{F}_\infty(G_2)}(v)$$
We will show that $\mathrm{lk}_{\mathcal{F}_\infty(G_1)}(v)\simeq*$. There are two possibilities: 
\begin{enumerate}
    \item $B_0\cong C_3$. Then $V(B_0)=\{v,v_1,v_2\}$ and $G_1=H_1\cup B_0\cup H_2$, with 
    $V(H_1)\cap V(B_0)=\{v_1\}$, $V(H_2)\cap V(B_0)=\{v_2\}$ and $V(H_1)\cap V(H_2)=\emptyset$.  Then, by Lemma \ref{lemlinkvertrian}, 
    $\mathrm{lk}_{\mathcal{F}_\infty(G_1)}(v)\simeq\hocolim(\mathcal{S})$ with $\mathcal{S}$ the diagram:
    $$\mathcal{F}_\infty(H_1)*\mathcal{F}_\infty(H_2-v_2)\longhookleftarrow \mathcal{F}_\infty(H_1-v_1)*\mathcal{F}_\infty(H_2-v_2)\longhookrightarrow \mathcal{F}_\infty(H_1-v_1)*\mathcal{F}_\infty(H_2)$$
    By construction, $G_1$ does not have saturated blocks, then $\delta(H_1-v_1)=1$ or it has a cut vertex of degree $2$. 
    Therefore $\mathcal{F}_\infty(H_1-v_1)\simeq*$. Analogously, $\mathcal{F}_\infty(H_2-v_2)\simeq*$. 
    From this, we get that $\hocolim(\mathcal{S})\simeq*$.
    \item $B_0\cong C_n$ with $n\geq4$. Let $v_1,v_2$ be the neighbors of $v$ in $B_0$ and take $H$ be the graph obtained
    from $G_1$ by erasing $v$ and adding the edge $v_1v_2$. Then 
    $$\mathrm{lk}_{\mathcal{F}_\infty(G_1)}(v)=\mathcal{F}_\infty(H)\simeq*,$$
    because $\mathcal{F}_\infty(H)$ has only one saturated block.
\end{enumerate}
Therefore $\mathrm{lk}_{\mathcal{F}_\infty(G)}(v)\simeq*$ and $\mathcal{F}_\infty(G)\simeq \mathcal{F}_\infty(G-v)$. 
If there is a non-saturated block which contains $v$, then 
$\delta(G-v)=1$ or there is a cut vertex of degree $2$, and therefore $\mathcal{F}_\infty(G)\simeq*$. Assume that there is no non-saturated 
block with $v$ among its vertices. Now, in $G-v$, all the remaining edges of the blocks that contain $v$ are bridges, so we can remove 
them, let $H$ be the graph thus obtained. If $B_0,B_1,\dots,B_{l-1}$ are the blocks that contain $v$, 
with $n_0,n_1,\dots,n_{l-1}$ their respective orders, then  $H=H_1+\cdots+H_r$
where 
$$r=\displaystyle\sum_{i=0}^{l-1}(n_i-1).$$
By inductive hypothesis, each $\mathcal{F}_\infty(H_i)$ is contractible or is homotopy equivalent to a sphere of dimension at least 
$|V(H_i)|-b(H_i)-1$. Then, $\mathcal{F}_\infty(H)$ is contractible or it has the homotopy type of a sphere of dimension at least
$$r-1+\sum_{i=1}^r\left(|V(H_i)|-b(H_i)-1\right)=n-1-(b(G)-l)-1=n-b(G)+l-2>n-b(G)-1.$$
\end{proof}

\subsection{Graph operations}
\subsubsection{Join of graphs}

Given two graphs $G$ and $H$ with disjoint vertex sets, we define their \textit{join} as the graph $G*H$ with 
$V(G*H)=V(G)\cup V(H)$ and 
$$E(G*H)=E(G)\cup E(H)\cup\{uv\colon \;u\in V(G)\mbox{ and }v\in V(H)\}.$$
It is well-known that $\mathcal{F}_0(G*H)=\mathcal{F}_0(G)\sqcup\mathcal{F}_0(H).$
\begin{lem}\label{lemjoingraf}
Let $G$ and $H$ be graphs with disjoint vertex sets with orders $n_1$ and $n_2$ respectively. Then:
\begin{enumerate}
    \item $\displaystyle\mathcal{F}_1(G*H)\simeq\mathcal{F}_1(G)\vee\mathcal{F}_1(H)\vee\bigvee_{n_1n_2-1}\mathbb{S}^1$
    \item If $\mathcal{F}_0(G)$ and $\mathcal{F}_0(H)$ are connected. Then, for all $d\geq2$
$$\mathcal{F}_d(G*H)\simeq\left(\bigvee_{n_2-1}\Sigma\mathrm{sk}_{_{d-1}}\mathcal{F}_0(G)\right)\vee\left(\bigvee_{n_1-1}\Sigma \mathrm{sk}_{_{d-1}}\mathcal{F}_0(H)\right)\vee\left(\bigvee_{(n_1-1)(n_2-1)}\mathbb{S}^2\right)\vee A\vee B$$
with $A=\mathcal{F}_d(G)\cup C(\mathrm{sk}_{_{d-1}} \mathcal{F}_0(G))$ and $B=\mathcal{F}_d(H)\cup C(\mathrm{sk}_{_{d-1}} \mathcal{F}_0(H))$
\end{enumerate}
\end{lem}
\begin{proof}
For $d=1$,
$$\mathcal{F}_1(G*H)=\mathcal{F}_1(G)\cup\mathcal{F}_1(H)\cup K_{n_1,n_2}.$$
Now $\mathcal{F}_1(G)\cap\mathcal{F}_1(H)\cap K_{n_1,n_2}=\mathcal{F}_1(G)\cap\mathcal{F}_1(H)=\emptyset$, therefore 
$\mathcal{F}_1(G*H)$ is homotopy equivalent to the homotopy pushout of 
$$X\longleftarrow\mathrm{sk}_0\mathcal{F}_1(G) \longrightarrow\mathcal{F}_1(G),$$
where $X$ is the homotopy pushout of 
$$\mathcal{F}_1(H)\longleftarrow\mathrm{sk}_0\mathcal{F}_1(H)\longrightarrow K_{n_1,n_2}.$$
Thus 
$$X\simeq\mathcal{F}_1(H)\vee\bigvee_{n_1(n_2-1)}\mathbb{S}^1.$$
From this the result follows.

For $d\geq2$,
$$\mathcal{F}_d(G*H)=\mathcal{F}_d(G)\cup\mathcal{F}_d(H)\cup K_1\cup K_2, $$
with $K_1=\bigcup_{u\in V(H)}\{u\}*\mathrm{sk}_{_{d-1}}\mathcal{F}_0(G)$ and 
$K_2=\bigcup_{u\in V(G)}\{u\}*\mathrm{sk}_{_{d-1}}\mathcal{F}_0(H)$. 
Now:
$$K_1\cong\bigvee_{n_2-1}\Sigma\mathrm{sk}_{_{d-1}}\mathcal{F}_0(G),$$
$$K_2\cong\bigvee_{n_1-1}\Sigma\mathrm{sk}_{_{d-1}}\mathcal{F}_0(H).$$
Taking $L_1=\mathcal{F}_d(G)$ and $L_2=\mathcal{F}_d(H)$, we have that
$$L_1\cap L_2=\emptyset,\;K_1\cap L_1=\mathrm{sk}_{_{d-1}}\mathcal{F}_0(G),\;K_2\cap L_2=\mathrm{sk}_{_{d-1}}\mathcal{F}_0(H),\;K_1\cap K_2\cong K_{n_1,n_2},$$
$$L_1\cap K_1\cap K_2=L_1\cap K_2\cong\bigvee_{n_1-1}\mathbb{S}^0,$$
$$L_2\cap K_2\cap K_1=L_2\cap K_1\cong\bigvee_{n_2-1}\mathbb{S}^0.$$
Taking $X=K_1\cup L_1$ and $Y=K_2\cup L_2$, we have that $\mathcal{F}_d(G*H)=X\cup Y$ and 
$X\cap Y=\left(L_1\cap K_2\right)\cup\left(L_2\cap K_1\right)\cup\left(K_1\cap K_2\right)=K_1\cap K_2$.
Therefore $\mathcal{F}(G*H,d)\simeq \hocolim(\mathcal{S})$ with 
$$\mathcal{S}\colon \;X\longhookleftarrow K_{n_1,n_2}\longhookrightarrow Y$$
Now, the inclusion $i\colon K_{n_1,n_2}\longhookrightarrow X$ is really the inclusion 
$K_{n_1,n_2}\longhookrightarrow K_1$, which is null-homotopic, and therefore $i$ is null-homotopic. In the same way we see 
that the inclusion in $Y$ is null-homotopic and that
$$\mathcal{F}_d(G*H)\simeq X\vee Y\vee\bigvee_{_{(n_1-1)(n_2-1)}}\mathbb{S}^2.$$
Now, $K_1\cap L_1=\mathrm{sk}_{_{d-1}}\mathcal{F}_0(G)$ and its inclusion in $K_1$ is null-homotopic, therefore we can compute the homotopy type of $X$ by pasting these two homotopy pushout squares:
\begin{equation*}
    \xymatrix{
    \mathrm{sk}_{_{d-1}}\mathcal{F}_0(G) \ar@{->}[r] \ar@{->}[d] & \ast \ar@{->}[r] \ar@{->}[d] & K_1 \ar@{->}[d] \\
    L_1 \ar@{->}[r] & L_1\cup C(\mathrm{sk}_{_{d-1}}\mathcal{F}_0(G)) \ar@{->}[r] & K_1\vee (L_1\cup C(\mathrm{sk}_{_{d-1}}\mathcal{F}_0(G))) \simeq X
    }
\end{equation*}
Now $L_1\cup C(\mathrm{sk}_{_{d-1}}\mathcal{F}_0(G))=A$. With an similar argument for $Y$ we arrive at the result.
\end{proof}

With the last lemma we can construct graphs for which $\mathcal{F}_\infty(G)$ have not the homotopy type of a wedge of spheres.
Let $K$ be a triangulation of the projective plane and let $H$ be the complement graph of the $1$-skeleton of the baricentric 
subdivision, then $\mathcal{F}_0(H)\cong \mathbb{RP}^2$ and $G=P_4*H$ is a graph such that $\mathcal{F}_d(G)$ has torsion for all $d\geq3$.

\begin{lem}\label{lemcono}
Let $G$ be a graphof order $n$, then for all $d\geq2$
$$\mathcal{F}_d(K_1*G)\simeq\mathcal{F}_d(G)\cup C(\mathrm{sk}_{_{d-1}}\mathcal{F}_0(G)),$$
and for $d=1$
$$\mathcal{F}_1(G)\simeq\mathcal{F}_1(G)\vee\bigvee_{n-1}\mathbb{S}^1$$
\end{lem}
\begin{proof}
The link of the apex vertex is $\mathrm{sk}_{_{d-1}}\mathcal{F}_0(G)$, thus the homotopy pushout square
\begin{equation*}
    \xymatrix{
    \mathrm{sk}_{_{d-1}}\mathcal{F}_0(G) \ar@{->}[r] \ar@{->}[d] & \ast  \ar@{->}[d]\\
    \mathcal{F}_d(G) \ar@{->}[r] & \mathcal{F}_d(G)\cup C(\mathrm{sk}_{_{d-1}}\mathcal{F}_0(G)) 
    }
\end{equation*}
computes $\mathcal{F}(K_1*G,d)$. Now, for $d=1$, $\mathcal{F}_1(G)$ is connected and $\mathrm{sk}_{0}\mathcal{F}_0(G)$
is the disjoint union of $n$ points, the result follows by Lemma \ref{homocolimpegado}. 
\end{proof}

\begin{theorem}
For the complete bipartite graph we have that $\mathcal{F}_0(K_{n,m})\simeq\mathbb{S}^0$, 
$$\mathcal{F}_1(K_{n,m})\simeq\bigvee_{nm-1}\mathbb{S}^1,$$
$$\mathcal{F}_d(K_{n,m})\simeq\bigvee_{(n-1)(m-1)}\mathbb{S}^2\vee\bigvee_{n\binom{m-1}{d}+m\binom{n-1}{d}}\mathbb{S}^d,$$
for $\infty>d\geq2$ and 
$$\mathcal{F}_\infty(K_{n,m})\simeq\bigvee_{(n-1)(m-1)}\mathbb{S}^2.$$
\end{theorem}
\begin{proof}
If $d=0$ is clear. The case $d=1$ is a particular case of Lemma \ref{lemjoingraf}. 
For $d\geq2$, by Lemma \ref{lemjoingraf} 
$$\mathcal{F}_d(K_{n,m})\simeq\left(\bigvee_{m-1}\Sigma\mathrm{sk}_{_{d-1}}\mathcal{F}_0(K_n^c)\right)\vee\left(\bigvee_{n-1}\Sigma \mathrm{sk}_{_{d-1}}\mathcal{F}_0(K_m^c)\right)\vee\left(\bigvee_{(n-1)(m-1)}\mathbb{S}^2\right)\vee A\vee B$$
with $A=\mathcal{F}_d(K_n^c)\cup C(\mathrm{sk}_{_{d-1}} \mathcal{F}_0(K_n^c))$ and $B=\mathcal{F}_d(K_m^c)\cup C(\mathrm{sk}_{_{d-1}} \mathcal{F}_0(K_m^c))$.

Now, for all $d,k,r$, 
$$\mathcal{F}_d(K_k^c)\cong\Delta^{k-1},\;\mathrm{sk}_r\mathcal{F}_d(K_k^c)\simeq\bigvee_{\binom{k-1}{r+1}}\mathbb{S}^r;$$
therefore
$$A\simeq\bigvee_{\binom{n-1}{d}}\mathbb{S}^d;\;B\simeq\bigvee_{\binom{m-1}{d}}\mathbb{S}^d,$$
from which we obtain the result. 
\end{proof}

\begin{cor}
Let $G_1,G_2,\dots,G_k$ be vertex disjoint graphs. For $d\geq1$, if $\mathcal{F}_d(G_i)\simeq *$ for all $i$, then
$$\mathcal{F}_d(G_1*G_2*\dots*G_k)\simeq\bigvee_{\frac{(k-1)(k-2)}{2}}\mathbb{S}^1\vee\bigvee_{i<j}\mathcal{F}_d(G_i*G_j)$$
\end{cor}
\begin{proof}
Let $V_i$ be the vertex set of $G_i$ and take $G=G_1*G_2*\dots*G_k$. 
If we take vertices from more than two sets of the partition, we will always have a cycle, and therefore each facet of the complex 
is contained in $V_i\cup V_j$ for some $i\neq j$. Then, taking $X_{ij}=\mathcal{F}_d(G\left[V_i\cup V_j\right])$ for 
$i<j$, we have that $\displaystyle \mathcal{F}_d(G)=\bigcup_{i<j}X_{ij}$ and we can define a bijection 
$\gamma\colon \{ij\colon \;i<j\}\longrightarrow E(K_k)$ such that the hypothesis of Lemma \ref{lemhomotopiagraf} are 
achieved. 
\end{proof}

As an immediate consequence we have the homotopy type for the multipartite graphs
\begin{cor}\label{multi}
For $d\geq1$,
$$\mathcal{F}_d(K_{n_1,\dots,n_k})\simeq\bigvee_{\frac{(k-1)(k-2)}{2}}\mathbb{S}^1\vee\bigvee_{i<j}\mathcal{F}_d(K_{n_i,n_j}).$$
\end{cor}

\begin{theorem}\label{kozlovciclos}\citep{kozlovdire}
$$\mathcal{F}_0(C_n)\simeq\left\lbrace\begin{array}{cc}
    \mathbb{S}^{r-1}\vee\mathbb{S}^{r-1} & \mbox{ if }n=3r \\
    \mathbb{S}^{r-1} & \mbox{ if }n=3r+1\\
    \mathbb{S}^{r} &\mbox{ if }n=3r+2
\end{array}
\right.$$
\end{theorem}

\begin{prop}
Let $W_{n+1}$ be the wheel on $n+1$ vertices, then 
$$\mathcal{F}_d(W_{n+1})\simeq\left\lbrace\begin{array}{cc}
\mathbb{S}^{3r-2}\vee\mathbb{S}^{r}\vee\mathbb{S}^{r}&\mbox{if }n=3r\\
\mathbb{S}^{3r-1}\vee\mathbb{S}^{r}&\mbox{if }n=3r+1\\
\mathbb{S}^{3r}\vee\mathbb{S}^{r+1}&\mbox{if }n=3r+2
\end{array}\right.$$
for $d>\lfloor\frac{n}{2}\rfloor-1$ and
$$\mathcal{F}_1(W_{n+1})\simeq\left\lbrace\begin{array}{cc}
    \displaystyle\bigvee_{3}\mathbb{S}^{2r-1}\vee\bigvee_{n-1}\mathbb{S}^1 & \mbox{ if } n=4r\\
    \displaystyle\mathbb{S}^{2r-1}\vee\bigvee_{n-1}\mathbb{S}^1 & \mbox{ if } n=4r+1\\
    \displaystyle\mathbb{S}^{2r}\vee\bigvee_{n-1}\mathbb{S}^1 & \mbox{ if } n=4r+2\\
    \displaystyle\mathbb{S}^{2r+1}\vee\bigvee_{n-1}\mathbb{S}^1 & \mbox{ if } n=4r+3\\
\end{array}\right.$$
\end{prop}
\begin{proof}
Since $\alpha(C_n)=\lfloor\frac{n}{2}\rfloor$, for $d>\lfloor\frac{n}{2}\rfloor-1$ we have that 
$\mathcal{F}_0(C_n)=\mathrm{sk}_{_{d-1}}\mathcal{F}_0(C_n)$. By Lemma \ref{lemcono},
$$\mathcal{F}_d(W_{n+1})\simeq \mathcal{F}_d(C_n)\cup C(\mathcal{F}_0(C_n)).$$
By Theorem \ref{kozlovciclos}, the inclusion of the intersection is null-homotopic, therefore
$$\mathcal{F}_d(W_{n+1})\simeq \mathcal{F}_\infty(C_n)\vee\Sigma\mathcal{F}_0(C_n)$$
For $d=1$, rhe result follows from Lemma \ref{lemcono} and Corollary \ref{corciclos1}.
\end{proof}

\subsubsection{Graph products}

\begin{prop}
$$\mathcal{F}_\infty\left(P_2\oblong P_k\right)\simeq\left\lbrace\begin{array}{cc}
    \mathbb{S}^{4r-1} & \mbox{if } k=3r \\
    * & \mbox{if } k=3r+1 \\
    \mathbb{S}^{4r+2} & \mbox{if } k=3r+2.
\end{array}
\right.$$
\end{prop}
\begin{proof}
By Theorem \ref{coneccuello}, $\mathcal{F}_\infty\left(P_2\oblong P_k\right)$ is simply connected. We will show 
that it has at most one non-trivial reduced homology group. The Alexander dual of 
$\mathcal{F}_\infty\left(P_2\oblong P_k\right)$ has as maximal simplicies the complements of $X_i=\{(i,1),(i+1,1),(i,2),(i+1,2)\}$ 
for $1\leq i\leq k-1$. Taking $U_i=X_i^c$ and $U$ the cover formed by these $U_i$, we have that
$$\mathcal{N}(U)\simeq\mathcal{F}_0^*(P_k).$$
It is standard that \citep{kozlovdire}:
$$\mathcal{F}_0(P_k)\simeq\left\lbrace\begin{array}{cc}
\mathbb{S}^{r-1} & \mbox{ if } k=3r \\
* & \mbox{ if } k=3r+1 \\
\mathbb{S}^{r} & \mbox{ if } k=3r+2.
\end{array}
\right.$$
Thus, by Theorem \ref{dualidadalexander}, $\mathcal{N}(U)$ has non-trivial reduced cohomology groups if $k=3r$ or $k=3r+2$, 
in which case the groups are in dimensions  
are $2(r-1)$ and $2r-1$ respectively. Therefore $\mathcal{F}_\infty\left(P_2\oblong P_k\right)$ is contractible if 
$k=3r+1$ and 
$$\tilde{H}_{q}\left(\mathcal{F}_\infty\left(P_2\oblong P_{3r}\right)\right)\cong\left\lbrace\begin{array}{cc}
\mathbb{Z} & \mbox{ if } q=4r-1 \\
0 &  \mbox{ if } q\neq4r-1,
\end{array}
\right.$$
$$\tilde{H}_{q}\left(\mathcal{F}_\infty\left(P_2\oblong P_{3r+2}\right)\right)\cong\left\lbrace\begin{array}{cc}
\mathbb{Z} & \mbox{ if } q=4r+2 \\
0 &  \mbox{ if } q\neq4r+2.
\end{array}
\right.$$
By Theorem \ref{cwsimplcon}, in these cases the complex is homotopy equivalent to a sphere of the desired dimension.
\end{proof}

It is known \cite{homotopygoyal} that 

$$\mathcal{F}_0(K_n\times K_m)\simeq\bigvee_{(n-1)(m-1)}\mathbb{S}^1.$$
Now we will show what happens for $d\geq1$.
\begin{prop}
$$\mathcal{F}_1(K_n\times K_m)\simeq\bigvee_{\frac{(nm-4)(n-1)(m-1)}{4}}\mathbb{S}^2$$
\end{prop}
\begin{proof}
We take $V(K_r)=[r]-\{0\}$ for any $r$.
We proceed by induction on $n$. For $n=1$, the result is clear. 
For $n=2$ we will prove it by induction on $m$. For $m=1,2$ it is clear and for $m=3$, $K_2\times K_3\cong C_6$. 
We take $m$ and assume that the result is true for any $k\leq m-1$.
Taking $v_i=(1,i)$ and $u_i=(2,i)$, we have that
$\mathrm{lk}(v_m)=X\cup Y$, where 
$Y=\mathcal{F}_1(K_2\times K_m)-N[v_m]$ and $X$ is the complex with facets $\{u_i,v_i,u_m\}$ for $i\geq m-1$. Then $X\simeq*$, as it is a cone with apex $u_m$, and $X\cap Y\cong K_{i,m}\simeq*$. Therefore, 
$$\mathrm{lk}(v_n)\simeq Y\cong\mathcal{F}_1(K_{1,m-1})\simeq\bigvee_{m-2}\mathbb{S}^1.$$
Taking $H=K_2\times K_m-v_m$, the link of $u_m$ in $\mathcal{F}_1(H)$ has as facets the simplex $\{u_1,\dots,u_{m-1}\}$ and 
the edges $\{u_i,v_i\}$ for $i\geq m-1$, therefore it is contractible and 
$$\mathcal{F}_1(H)\simeq\mathcal{F}_1(H-u_m)\cong\mathcal{F}_1(K_2\times K_{m-1})\simeq\bigvee_{\frac{(m-2)(m-3)}{2}}\mathbb{S}^2,$$
from which the result follows. 

Now assume the result is true for $K_r\times K_m$ for all $r\leq n-1$. Take 
$v_i=(n,i)$, $G_0=K_n\times K_m$, $G_i=G_{i-1}-v_i$ for $i\geq1$,
$X_{j,k}^i=|\{(j,k),(j,i),(n,k)\}|$ for $k\geq i+1$ and $j\leq n-1$, 
$X_{j,k}^i=|\{(j,k),(j,i)\}|$ for $k\leq i-1$ and $j\leq n-1$,
$$X^i=\bigcup_{k\neq i,\;j\leq n-1}X_{j,k}^i$$
and $Y^i=\mathcal{F}_1(G_{i-1}-N[v_i])$. Then, taking $L_{i}$ the link of $v_{i}$ in 
$\mathcal{F}_1(G_{i-1})$, we have that
$$L_i=X^i\cup Y^i.$$
Now, in $X^i$, the vertices $(j,k)$ with $j\leq n-1$ and $k\neq i$ are only in one facet and can be erased, therefore 
$X^i$ is homotopy equivalent to the subcomplex with maximal facets $\{(j,i),(n,k)\}$ with $k\geq i+1$ and $j\leq n-1$, which is isomorphic 
to $K_{n-1,m-i}$. Because $X^i\cap Y_i$ is isomorphic to this subcomplex, we have that 
$$L_i\simeq Y^i\cong \mathcal{F}_1(K_{n-i,m-1})\simeq\bigvee_{(m-1)(n-i)-1}\mathbb{S}^1$$
for $i\leq n-1$. Now, $L_n\simeq Y^n\simeq*$, therefore
$$\mathcal{F}_1(G_{n-1})\simeq\mathcal{F}_1(G_n)\cong\mathcal{F}_1(K_{n-1}\times K_{m})\simeq\bigvee_{\frac{((n-1)m-4)(n-2)(m-1)}{4}}\mathbb{S}^2.$$
From this we have that 
$$\mathcal{F}_1(G_0)\simeq\mathcal{F}_1(K_{n-1}\times K_{m})\vee\Sigma Y^1\vee\Sigma Y^2\vee\cdots\vee\Sigma Y^{n-1}.$$
Now $\Sigma Y^1\vee\Sigma Y^2\vee\cdots\vee\Sigma Y^{n-1}$ is homotopy equivalent to the wedge of
$$\sum_{i=1}^{m-1}i(n-1)-1=\frac{(n-1)m(m-1)}{2}-(m-1)$$
copies of the $2$-sphere. Since
$$\frac{((n-1)m-4)(n-2)(m-1)}{4}=\sum_{i=1}^{n-2}\frac{im(m-1)}{2}-(m-1),$$
we have that $\mathcal{F}_1(K_n\times K_m)$ is homotopy equivalent to the wedge of 
$$\sum_{i=1}^{n-1}\frac{im(m-1)}{2}-(m-1)=\frac{(nm-4)(n-1)(m-1)}{4}$$
$2$-spheres.
\end{proof}

\begin{figure}
\centering
\subfigure[]{\begin{tikzpicture}[line cap=round,line join=round,>=triangle 45,x=1.0cm,y=1.0cm]
\clip(-1.5,-0.2) rectangle (2,3.2);
\draw (-1,3)-- (1,2);
\draw (-1,3)-- (1,0);
\begin{scriptsize}
\fill [color=black] (-1,3) circle (1.5pt);
\fill [color=black] (1,3) circle (1.5pt);
\fill [color=black] (1,2) circle (1.5pt);
\draw [color=black] (1.2,2) node {$1$};
\draw [color=black] (1,1.5) node {$\vdots$};
\draw [color=black] (1,1) node {$\vdots$};
\draw [color=black] (1,0.5) node {$\vdots$};
\fill [color=black] (1,0) circle (1.5pt);
\draw [color=black] (1.5,0) node {$d+1$};
\end{scriptsize}
\end{tikzpicture}\label{simd2}}
\subfigure[]{\begin{tikzpicture}[line cap=round,line join=round,>=triangle 45,x=1.0cm,y=1.0cm]
\clip(-1.5,-0.2) rectangle (2,3.2);
\draw (-1,3)-- (1,2);
\draw (-1,3)-- (1,0);
\begin{scriptsize}
\fill [color=black] (-1,3) circle (1.5pt);
\fill [color=black] (1,2) circle (1.5pt);
\draw [color=black] (1.2,2) node {$1$};
\draw [color=black] (1,1.5) node {$\vdots$};
\draw [color=black] (1,1) node {$\vdots$};
\draw [color=black] (1,0.5) node {$\vdots$};
\fill [color=black] (1,0) circle (1.5pt);
\draw [color=black] (1.5,0) node {$d+1$};
\end{scriptsize}
\end{tikzpicture}\label{simd1}}
\caption{}
\end{figure}

 \begin{lem}
 For $d\geq2$, $\mathcal{F}_{d+1}(K_2\times K_n)\simeq\mathcal{F}_d(K_2\times K_n)$
 \end{lem}
\begin{proof}
We know that $\mathcal{F}_d(K_2\times K_n)$ is simply connected for all $d\geq2$, because 
$\mathcal{F}_1(K_2\times K_n)$ is a wedge of $2$-spheres.
We will show that $H_q(\mathcal{F}_{d+1}(K_2\times K_n),\mathcal{F}_{d}(K_2\times K_n))\cong0$ for all $q$. 
We know that $H_q(\mathcal{F}_{d+1}(K_2\times K_n),\mathcal{F}_{d}(K_2\times K_n))\cong0$ for all $q\leq d$. 
For $q\geq d+3$, for any $q$-simplex $\sigma$ of $\mathcal{F}_{d+1}(K_2\times K_n)$, we can partition its 
vertices in two sets $V_1,V_2$ such that all the vertices in $V_i$ are of the form $(i,j)$ for some $j$. Next we show that
$|V_1|=0$ or $|V_2|=0$. If not, we can assume that
$$|V_1|\leq\left\lfloor\frac{d+3}{2}\right\rfloor\leq\left\lceil\frac{d+3}{2}\right\rceil\leq|V_2|$$
therefore $|V_2|\geq3$; there are several cases: 
\begin{itemize}
    \item If $|V_1|=1$, then $|V_2|\geq d+3$ and the vertex of $V_1$ has degree at least $d+2$, which can not 
    happen.
    \item If $|V_1|=2$, then $|V_2|\geq d+2$ and there will be at least two vertices of $V_2$ such their second 
    coordinates are different from those of the vertices of $V_1$; therefore there will be an induced $4$-cycle 
    in the vertices of $\sigma$, which can not happen. 
    \item If $|V_1|\geq3$, because $|V_2|\geq3$, there will be an induced $4$-cycle or an induced 
    $6$-cycle in the vertices of $\sigma$, which can not happen.
\end{itemize}
Therefore $|V_1|=0$ or $|V_2|=0$ and $\sigma$ is a simplex of $\mathcal{F}_{d}(K_2\times K_n)$. From this, 
we have that $H_q(\mathcal{F}_{d+1}(K_2\times K_n),\mathcal{F}_{d}(K_2\times K_n))\cong0$ for all $q\geq d+3$.

For $q=d+2$, the only $q$-simplices of ${F}_{d+1}(K_2\times K_n)$ which are not simplices of $\mathcal{F}_{d}(K_2\times K_n)$
are of the form $|V_1|=1$ and $|V_2|=d+2$ (or vice versa), where the only vertex of $V_1$ is adjacent to all but one vertex of 
$V_2$ (Figure \ref{simd2}). For $q=d+1$, the only $q$-simplices of ${F}_{d+1}(K_2\times K_n)$ which are not simplices 
of $\mathcal{F}_{d}(K_2\times K_n)$ are of the form $|V_1|=1$ and $|V_2|=d+1$ (or vice versa), where the only vertex of $V_1$ is adjacent 
to all the vertices of $V_2$ (Figure \ref{simd1}). From all this, we get that there are no relative $d+2$-cycles and 
that all of the relative $d+1$-cycles are images of some relative $d+2$-boundary. Therefore the remaining two 
relative homology groups are also trivial. 

From all this we have that the inclusion $\mathcal{F}_{d}(K_2\times K_n)\longhookrightarrow\mathcal{F}_{d+1}(K_2\times K_n)$
induces an isomorphism for all homology groups between simply connected complexes, therefore 
$\mathcal{F}_{d+1}(K_2\times K_n)\simeq\mathcal{F}_{d}(K_2\times K_n)$.
\end{proof}

\begin{prop}
For $d\geq2$,
$$\mathcal{F}_d(K_2\times K_n)\simeq\bigvee_{\binom{n}{3}}\mathbb{S}^4\vee\bigvee_{\binom{n-1}{3}}\mathbb{S}^3.$$
\end{prop}
\begin{proof}
We only need to show it for $d=2$.
The result is clear for $n=1,2,3$. Assume $n\geq4$. Taking 
$\displaystyle k=\binom{n}{3}$, let $X_1,\dots,X_k$ be the subcomplexes of $\mathcal{F}_2(K_2\times K_n)$ corresponding 
to all the induced $6$-cycles. Then $X_i\cong\mathbb{S}^4$. The other facets of $\mathcal{F}_2(K_2\times K_n)$, besides the ones 
in some $X_i$, are $\{1\}\times\underline{n}$ and $\{2\}\times\underline{n}$. Then 
$$\mathcal{F}_2(K_2\times K_n)=X_1\cup X_2\cup\cdots\cup X_k\cup Y_1\cup Y_2$$
where $Y_1=\mathcal{P}(\{1\}\times\underline{n})$ and 
$Y_2=\mathcal{P}(\{2\}\times\underline{n})$. Now we will calculate the homology of $\mathcal{F}_2(K_2\times K_n)$ 
using the Mayer-Vietoris spectral sequence. Taking $U=\{X_1,X_2,\dots,X_k,Y_1,Y_2\}$ and $\mathcal{U}=\mathcal{N}(U)$, 
the first page of the sequence is 
\begin{equation*}
    \xymatrix{
    \mathbb{Z}^k \ar@{<-}[r] & 0 \ar@{<-}[r] & 0 \ar@{<-}[r] & 0 \ar@{<-}[r] & 0\\
    0 \ar@{<-}[r] &  0 \ar@{<-}[r] & 0 \ar@{<-}[r] & 0 \ar@{<-}[r] & 0\\
    0 \ar@{<-}[r] &  0 \ar@{<-}[r] & 0 \ar@{<-}[r] & 0 \ar@{<-}[r] & 0\\
    0 \ar@{<-}[r] &  0 \ar@{<-}[r] & 0 \ar@{<-}[r] & 0 \ar@{<-}[r] & 0\\
    C_0(\mathcal{U}) \ar@{<-}[r] & C_1(\mathcal{U}) \ar@{<-}[r] & C_2(\mathcal{U}) \ar@{<-}[r] & C_3(\mathcal{U}) \ar@{<-}[r] & C_4(\mathcal{U})
    }
\end{equation*}
Because the nerve of $X_1,X_2,\dots,X_k$ is isomorphic to the nerve of $2$-simplices of $\mathrm{sk}_2\Delta^{n-1}$, 
and $\mathcal{U}$ is isomorphic to the suspension of this nerve, we have that the second page is
\begin{equation*}
    \xymatrix{
    \mathbb{Z}^k \ar@{<-}[rrd] & 0 \ar@{<-}[rrd] & 0 \ar@{<-}[rrd] & 0  & 0\\
    0 \ar@{<-}[rrd] &  0 \ar@{<-}[rrd] & 0 \ar@{<-}[rrd] & 0  & 0\\
    0 \ar@{<-}[rrd] &  0 \ar@{<-}[rrd] & 0 \ar@{<-}[rrd] & 0  & 0\\
    0 \ar@{<-}[rrd] &  0 \ar@{<-}[rrd] & 0 \ar@{<-}[rrd] & 0 & 0\\
    \mathbb{Z}  & 0 & 0 & \mathbb{Z}^{r} & 0
    }
\end{equation*}
where $r=\binom{n-1}{3}$. From this we have that $E_{p,q}^\infty=E_{p,q}^2$. Therefore 
$$\tilde{H}_q(\mathcal{F}_2(K_2\times K_n))\cong\left\lbrace
\begin{array}{cc}
    \mathbb{Z}^k & \mbox{ if } q=4 \\
    \mathbb{Z}^r & \mbox{ if } q=3 \\
    0 & \mbox{ if } q\neq4,3 \\
\end{array}\right.$$
Therefore, because  $\mathcal{F}_1(K_2\times K_n)$ is simply connected, $\mathcal{F}_2(K_2\times K_n)$ is a simply connected complex which 
satisfies the hypothesis of Theorem \ref{gradconse}, from which we see that is has the desired homotopy type.
\end{proof}

\begin{theorem}
For $d\geq2$,
$$\mathcal{F}_d(K_n\times K_m)\simeq\bigvee_a\mathbb{S}^4\vee\bigvee_{b+c}\mathbb{S}^3,$$
where $a=\binom{m}{2}\binom{n}{3}+\binom{n}{2}\binom{m}{3}$, 
$b=\binom{m}{2}\binom{n-1}{3}+\binom{n}{2}\binom{m-1}{3}$ and $c=\binom{n-1}{2}\binom{m-1}{2}$.
\end{theorem}
\begin{proof}
In $\mathcal{F}_d(K_n\times K_m)$ the facets have their vertices contained in two rows or two columns, otherwise they will have 
a cycle. Then, taking the subgraphs
$$H_{i,j}=K_n\times K_m[\{(k,l)\colon\;l=i\mbox{ or }l=j\}],$$
$$G_{i,j}=K_n\times K_m[\{(k,l)\colon\;k=i\mbox{ or }k=j\}],$$
and the complexes $X_{i,j}=\mathcal{F}_d(H_{i,j})$, $Y_{i,j}=\mathcal{F}_d(G_{i,j})$, we have that
$$\mathcal{F}_d(K_n\times K_m)=\bigcup_{e\in E(K_m)}X_e\cup\bigcup_{e\in E(K_n)}Y_e$$
From the last Proposition we know that  
$$X_e\simeq\bigvee_{\binom{n}{3}}\mathbb{S}^4\vee\bigvee_{\binom{n-1}{3}}\mathbb{S}^3$$
$$Y_e\simeq\bigvee_{\binom{m}{3}}\mathbb{S}^4\vee\bigvee_{\binom{m-1}{3}}\mathbb{S}^3$$
Taking the Mayer-Vietoris spectral sequence, the first 
page looks like
\begin{equation*}
    \xymatrix{
    \mathbb{Z}^a \ar@{<-}[r] & 0 \ar@{<-}[r] & 0 \ar@{<-}[r] & 0 \ar@{<-}[r] & 0\\
    \mathbb{Z}^b \ar@{<-}[r] &  0 \ar@{<-}[r] & 0 \ar@{<-}[r] & 0 \ar@{<-}[r] & 0\\
    0 \ar@{<-}[r] &  0 \ar@{<-}[r] & 0 \ar@{<-}[r] & 0 \ar@{<-}[r] & 0\\
    0 \ar@{<-}[r] &  0 \ar@{<-}[r] & 0 \ar@{<-}[r] & 0 \ar@{<-}[r] & 0\\
    C_0(\mathcal{U}) \ar@{<-}[r] & C_1(\mathcal{U}) \ar@{<-}[r] & C_1(\mathcal{U}) \ar@{<-}[r] & C_3(\mathcal{U}) \ar@{<-}[r] & C_4(\mathcal{U}) 
    }
\end{equation*}
Where $\mathcal{U}$ is the nerve of the cover, 
$a=\binom{m}{2}\binom{n}{3}+\binom{n}{2}\binom{m}{3}$ and $b=\binom{m}{2}\binom{n-1}{3}+\binom{n}{2}\binom{m-1}{3}$. 
Now, $\mathcal{U}$ is isomorphic to the join of the nerve of the $X'$s with the nerve of the $Y'$s, which are 
homotopy equivalent to $K_m$ and $K_n$ respectively, therefore $\mathcal{U}\simeq\bigvee_{c}\mathbb{S}^3$ with 
$c=\binom{n-1}{2}\binom{m-1}{2}$. From all this, we have that the second page of the sequence is
\begin{equation*}
    \xymatrix{
    \mathbb{Z}^a \ar@{<-}[rrd] & 0 \ar@{<-}[rrd] & 0 \ar@{<-}[rrd] & 0 & 0 \\
    \mathbb{Z}^b \ar@{<-}[rrd] &  0 \ar@{<-}[rrd] & 0 \ar@{<-}[rrd] & 0 & 0 \\
    0 \ar@{<-}[rrd] &  0 \ar@{<-}[rrd] & 0 \ar@{<-}[rrd] & 0 & 0 \\
    0 \ar@{<-}[rrd] &  0 \ar@{<-}[rrd] & 0 \ar@{<-}[rrd] & 0 & 0 \\
    \mathbb{Z}  & 0 & 0 & \mathbb{Z}^c & 0 
    }
\end{equation*}
Therefore $E_{p,q}^\infty=E_{p,q}^2$ and
$$\tilde{H}_q(\mathcal{F}_d(K_n\times K_m))\cong\left\lbrace
\begin{array}{cc}
    \mathbb{Z}^a & \mbox{ if } q=4 \\
    \mathbb{Z}^{b+c} & \mbox{ if } q=3 \\
    0 & \mbox{ if } q\neq4,3 \\
\end{array}\right.$$
As in the proof of the last theorem, we have a simply connected complex which satisfies the hypothesis of Theorem \ref{gradconse}.
\end{proof}

In \citep{indcomplcartprod} it was shown that
$$\mathcal{F}_0(K_2\times K_m \times K_n)\simeq\bigvee_{\frac{(n-1)(m-1)(nm-2)}{2}}\mathbb{S}^3.$$
For $d\geq1$, because $K_2\times K_2\cong K_2\sqcup K_2$ we have the following corollary
\begin{cor}
For $d\geq1$,
    $$\mathcal{F}_d(K_2\times K_2\times K_n)\simeq\left\lbrace\begin{array}{cc}
      \displaystyle\bigvee_{\frac{(n-2)^2(n-1)^2}{4}}\mathbb{S}^5  &  d=1\\
        \displaystyle\bigvee_{\binom{n}{3}^2}\mathbb{S}^9\vee\bigvee_{2\binom{n}{3}\binom{n-1}{3}}\mathbb{S}^8\vee\bigvee_{\binom{n-1}{3}^2}\mathbb{S}^7 & d\geq2
    \end{array}\right.$$
\end{cor}

\begin{que}
What is the homotopy type of $K_2\times K_m \times K_n$ for $d\geq1$?
\end{que}

\section{Connections to other results}
In this section we will give an upper bound for $\nabla_d$ using rational cohomology. We defined $\nabla_d(G)$ 
as the minimum number of vertices needed to erase from $G$ to obtain a forest of maximum degree at most $d$, and it 
is equal to the minimum number of vertices such that its complement is a simplex of $\mathcal{F}_d(G)$ --it can also can be seen as the 
minimum number of a transversal of the minimal non-simplices of $\mathcal{F}_d(G)$. Recall that if $G$ has $n$ vertices, then 
$n=\nabla_d(G)+\mathrm{dim}(\mathcal{F}_d(G))+1=\nabla_d(G)+t_d(G)$.

We start by giving some definitions. 
For a simplicial complex $K$, we define
$$\mathrm{hrk}(K,\mathbb{Q})=\sum_{i=-1}\mathrm{dim}\left(\tilde{H}^{i}(K,\mathbb{Q})\right),$$
where $\tilde{H}^{-1}(K)$ is $\mathbb{Q}$ if $K=\{\emptyset\}$, and it is $0$ otherwise; we take $K[\emptyset]=\{\emptyset\}$ if 
$K$ is non-empty.  

It is worth mentioning that there is a lower bound for the decycling number using the rational homology of the independence complex.
\begin{theorem}\citep{engstromupperbounds}
For any graph $G$
$$|\tilde{\chi}(\mathcal{F}_0(G))|\leq\sum_{i=0}\mathrm{dim}\left(\tilde{H}_i(\mathcal{F}_0(G),\mathbb{Q})\right)\leq2^{\nabla(G)}$$
\end{theorem}

The following theorem was proved in the context of the toral rank conjecture for the moment angle complex of a simplicial complex
\citep{caomobius,ustinovtoralrank}. We only need the following inequality and will not define the moment-angle complex nor state the 
toral rank conjecture (see \citep{panovtorictopology}).

\begin{theorem}\citep[Theorem 4.8.9]{panovtorictopology}\label{toralrank}
For any simplicial complex $K$ on $n$ vertices of dimension $k-1$
$$2^{n-k}\leq\sum_{S\subseteq V(K)}\mathrm{hrk}(K[S],\mathbb{Q})$$
\end{theorem}

Since the dimension of  $\mathcal{F}_d(G)$ is $t_d(G)-1$ and $\nabla_d(G)=n-t_d(G)$, Theorem \ref{toralrank} implies the following result.

\begin{prop}\label{propdecy}
For any graph $G$
$$\nabla_d(G)\leq\mathrm{log}_2\left(\sum_{S\subseteq V(G)}\mathrm{hrk}\left(\mathcal{F}_d(G)[S],\mathbb{Q}\right)\right)$$
\end{prop}

As a consequence, we get a lower bound for the independence number.
\begin{cor}
For any graph of order $n$
$$\alpha(G)\geq n-\mathrm{log}_2\left(\sum_{S\subseteq V(G)}\mathrm{hrk}\left(\mathcal{F}_0(G)[S],\mathbb{Q}\right)\right)$$    
\end{cor}
For $d=\infty$, we get the following.
\begin{cor}
For any graph $G$ of order $n$
$$\nabla(G)\leq\mathrm{log}_2\left(1+\sum_{r=g(G)}^n\sum_{S\in\binom{V(G)}{r}}\mathrm{hrk}\left(\mathcal{F}_\infty(G)[S],\mathbb{Q}\right)\right)$$
\end{cor}
Notice that if $G$ is a forest, the last bound is $0$, so in this case the bound is tight. We now provide some examples to
Proposition \ref{propdecy}. In the first three examples the bound also will be tight.
\begin{exa}
For $M_q$, the graph form by $q$-disjoint edges, we have that $\mathcal{F}_0(M_q)\cong\mathbb{S}^{q-1}$ and for any vertex set $S$ we have
that $\mathcal{F}_0(M_q[S])\cong\mathbb{S}^{\frac{|S|}{2}-1}$ if $M_q[S]\cong M_{|S|}$ and $\mathcal{F}_0(M_q[S])\simeq*$ otherwise. Then
$\tau(M_q)\leq q$. It is easy to see that in fact $\tau(M_q)=q$.  
\end{exa}

\begin{exa}
Take $G_r$ as the disjoint union of $r$ cycles of any length and $H_r$ as the graph obtain from $G_r$ by adding a new vertex $v$ 
which will be adjacent to one vertex of each cycle. Then $\mathcal{F}_\infty(G_r)$ has the homotopy type of a sphere and 
$\mathcal{F}_\infty(H_r)$ is contractible, but for both graphs the upper bound is $r$ which is the decycling number of both graphs.
\end{exa}

\begin{exa}
For $C_n$ and $d\geq2$ we have that 
$$\sum_{S\subseteq V(C_n)}\mathrm{hrk}(\mathcal{F}_d(C_n)[S],\mathbb{Q})=2=2^{\nabla(C_n)}.$$
\end{exa}

\begin{exa}
For $K_{n,m}$ and $n,m\geq2$ we get
$$\mathrm{log}_2\left(1+\sum_{i=2}^n\sum_{j=2}^m(i-1)(j-1)\binom{n}{i}\binom{m}{j}\right)\geq\nabla(K_{n,m})=\min\{n-1,m-1\}$$
If $n=m=2$ the bound is tight, but it is not tight if $m\geq3$ and $n=2$ for example.
\end{exa}

Recall that a graph $G$ is called \textit{ternary} if all its induced cycles have length not divisible by $3$. The \textit{Fibonacci number} 
of a graph $G$ is defined as the number of independent sets in $G$, that we denote by $NI(G)$. Thus 
$$NI(G)=\sum_{i=-1}^{\alpha(G)-1}f_i\left(\mathcal{F}_0(G)\right)$$
We now give an upper bound for the Fibonacci numbers of ternary graphs. This bound is given in terms of combinatorial parameters only, while 
the proof makes use of algebraic topological arguments.

\begin{prop}\label{fibotern}
If $G$ is a ternary graph, then
$$NI(G)\leq2^{n}-2^{\tau(G)}+1$$
\end{prop}
\begin{proof}
Let $G$ be a ternary graph of order $n$. 
It is known that a graph is ternary if and only if, for any induced subgraph, its independence complex is contractible or has the homotopy 
type of a sphere \citep{kimternary}. Thus 
$$2^{\tau(G)}\leq\sum_{S\subseteq V(K)}\mathrm{hrk}(K[S],\mathbb{Q})\leq2^n+1-NI(G)$$
The result now follows easily. 
\end{proof}

Notice that a forest is a ternary graph, then we obtain the following corollary.
\begin{cor}
Let $G$ be a forest of order $n$, then:
\begin{enumerate}
    \item\citep{linfibo} $NI(G)\leq2^{n}$.
    \item If $G$ is a tree of order $n\geq2$ and $\mathrm{diam}(G)=r$, then $NI(G)\leq2^{n}-2^{\left\lfloor\frac{r}{2}\right\rfloor}+1$.
\end{enumerate}
\end{cor}
\begin{proof}
The first part is obvious from the bound in Proposition \ref{fibotern}. The second part follows from the fact that if $G$
is a tree of diameter $r$, then $\tau(G)\geq\left\lfloor\frac{r}{2}\right\rfloor$.
\end{proof}
Notice that the first bound is tight as the complement of the complete graph shows, thus the bound of 
Proposition \ref{fibotern} is tight. 

\begin{que}
Are there other families of ternary graphs for which the bound of Proposition \ref{fibotern} is tight?
\end{que}

As a final remark, note that the bound of Proposition \ref{propdecy} can be generalized to transversal numbers of hypergraphs.
If $H$ is an hypergraph such that it does not have an hyperedge contain in another hyperedge, then the set of hyperedges can be seen 
as the set of minimal non-simplicies of a simplicial complex. The dimension of this complex will be $\alpha(H)-1$ and $n-\alpha(H)$ is the 
transversal number of the hypergraph, where $n$ is its order. Thus we obtain the corresponding bound for the transversal number.
 
\bibliographystyle{acm}
\bibliography{complejobosques}




\end{document}